\documentclass[10pt]{article}
\usepackage[T1]{fontenc}
\usepackage{amsmath,amsfonts,amsthm,mathrsfs,amssymb}
\usepackage{cite}
\usepackage{multirow}
\usepackage{graphicx,float}
\usepackage{subfigure}
\usepackage{placeins}
\usepackage{color}
\usepackage{indentfirst}
\usepackage[bookmarks=true]{hyperref}
\numberwithin{equation}{section}
\newcommand{\R}{{\mathbb R}}

\newcommand{\be}{\begin{eqnarray}}
\newcommand{\ben}{\begin{eqnarray*}}
\newcommand{\en}{\end{eqnarray}}
\newcommand{\enn}{\end{eqnarray*}}

\newcommand{\pa}{\partial}

\newtheorem{theorem}{Theorem}[section]

\newtheorem{remark}[theorem]{Remark}

\definecolor{rot}{rgb}{1.000,0.000,0.000}
\begin{document}
\renewcommand{\theequation}{\arabic{section}.\arabic{equation}}
\begin{titlepage}
  \title{On the hyper-singular boundary integral equation methods for dynamic poroelasticity: three dimensional case}

\author{Lu Zhang\thanks{School of Mathematical Sciences, University of Electronic Science and Technology of China, Chengdu, Sichuan 611731, China. Email: {\tt luzhang@std.uestc.edu.cn}}\;,
Liwei Xu\thanks{School of Mathematical Sciences, University of Electronic Science and Technology of China, Chengdu, Sichuan 611731, China. Email: {\tt xul@uestc.edu.cn}}\;,
Tao Yin\thanks{LSEC, Institute of Computational Mathematics and Scientific/Engineering Computing, Academy of Mathematics and Systems Science, Chinese Academy of Sciences, Beijing 100190, China. Email:{\tt yintao@lsec.cc.ac.cn}}}
\end{titlepage}
\maketitle

\begin{abstract}
In our previous work [SIAM J. Sci. Comput. 43(3) (2021) B784-B810], an accurate hyper-singular boundary integral equation method for dynamic poroelasticity in two dimensions has been developed. This work is devoted to studying the more complex and difficult three-dimensional problems with Neumann boundary condition and both the direct and indirect methods are adopted to construct combined boundary integral equations. The strongly-singular and hyper-singular integral operators are reformulated into compositions of weakly-singular integral operators and tangential-derivative operators, which allow us to prove the jump relations associated with the poroelastic layer potentials and boundary integral operators in a simple manner. Relying on both the investigated spectral properties of the strongly-singular operators, which indicate that the corresponding eigenvalues accumulate at three points whose values are only dependent on two Lam\'e constants, and the spectral properties of the Calder\'on relations of the poroelasticity, we propose low-GMRES-iteration regularized integral equations. Numerical examples are presented to demonstrate the accuracy and efficiency of the proposed methodology by means of a Chebyshev-based rectangular-polar solver.

\end{abstract}

\section{Introduction}
\label{sec1}

The dynamic poroelastic problems describing the physical behavior of the wave propagation in the elastic solid and the interstitial fluid can be found in many fields of applications such as petroleum industry, materials science, soil mechanics and biomechanics, etc. In accordance to Biot's theory~\cite{B41,B55,B561,B562,B563,CBB91,B00}, the dynamic poroelastic problems can be modeled by the coupled equations of the pore pressure and the solid displacement field, and the targeted degrees of freedom can be changed~\cite{S09}. For the numerical solutions of such kind of wave scattering problems, it is known that the boundary integral equation (BIE) methods~\cite{CK98,HW08} take advantages over the volumetric discretization methods~\cite{DR93,DE96,LS98,XOX19} in the sense of dimensions reduction, discretization of boundary and natural satisfactory of radiation condition, while in particular, the volumetric methods requires introducing appropriate artificial boundary conditions, such as absorbing boundary conditions or perfectly matched layers for the treatment of problems on unbounded domains. As a continuation of our previous work~\cite{ZXY21} for the two-dimensional poroelastic scattering problems, this work is devoted to proposing efficient BIE methods for solving the three-dimensional problems~\cite{CD95,MB89,MS11,MS12,S011,S012,SSU09} with Neumann boundary condition and it requires more complex technical investigations of the poroelastic boundary integral operators (BIOs).

In the classical BIE theory, both the direct methods based on Green's formula and the indirect methods based on potential theory have been extensively discussed. In practice, the combined boundary integral equations (CBIEs) resulting from a combination of single-layer and double-layer BIOs  (for Dirichlet case) or a combination of double-layer and hyper-singular BIOs (for Neumann case) are generally employed to avoid the influence of possible eigenfrequencies. In this work, we employ both the direct method and the indirect method to construct two types of CBIEs for solving the three-dimensional poroelastic problems with Neumann boundary condition. As mentioned in~\cite{ZXY21},  it still remains open to prove the unique solvability of the CBIEs for the poroelastic scattering problem, but these BIEs still can provide efficient numerical tools for the solutions of the problems imposed on unbounded domains.
Then analogous to the two-dimensional case~\cite{ZXY21}, the following three issues should be addressed:
\begin{itemize}
\item[(i).] {\it The jump relations between the layer-potentials and the BIOs in dynamic poroelasticity case are not easy to be observed.}
\item[(ii).] {\it The double-layer operators $K, K'$ (strongly-singular) and hyper-singular operator $N$ are well defined in the sense of Cauchy principle value and Hadamard finite part~\cite{HW08}, respectively. Then it requires appropriate solvers for the accurate evaluation of these operators.}
\item[(iii).] {\it It is known that the eigenvalues of the hyper-singular operator accumulate at infinity and as a result, solving the CBIEs by means of Krylov-subspace iterative solvers, such as GMRES, generally requires a relatively large number of iterations for the convergence of numerical solution. Then low-GMRES-iteration integral formulations are highly desirable.}
\end{itemize}

To resolve the first and second issues, it is necessary to take a comprehensive study on the single kernels of the poroelastic BIOs. To reduce/transform the singularities, some methodologies, for instance, adding-and-subtracting appropriate terms and regularization using integration-by-parts, have been discussed in open literatures. Inspired by the idea of reformulating the acoustic/Laplace hyper-singular integral operator into a combination of weakly-singular integral operators and tangential derivatives~\cite{HW08,N01}, a novel regularization technique using G\"unter derivative and Stokes formulas has been developed for the elastic and thermoelastic problems~\cite{BXY17,BXY191,L14,YHX17}. In two-dimensions, the G\"unter derivative can be simplified as the classical tangential derivative multiplied by a constant matrix, and the regularized formulations for two-dimensional poroelastic BIOs have been investigated in \cite{ZXY21}. But the three-dimensional G\"unter derivative is more complex. Although the thermoelastic problem~\cite{KGBB79} takes a similar Biot's model as the poroelastic problem, the results presented in \cite{BXY191} can not be extended to the three-dimensional poroelastic case trivially regarding to the more complicated coupled boundary operator and its adjoint, see Section~\ref{sec:3}. It is proved in Theorems~\ref{regK}-\ref{regN4} that the three-dimensional strongly-singular and hyper-singular poroelastic integral operators can be re-expressed in terms of multiple weakly-singular integral operators and tangential-derivative operators. Compared with the formulations given in~\cite{MS11,MS12}, the derived regularized expressions in this work take simpler forms, and as a consequence, the jump relations between the layer-potentials and the BIOs in dynamic poroelasticity case can be proved in an extremely simple manner, see Theorem~\ref{jump} and its proof in Section~\ref{sec:4.3}. In addition, owing to the new regularized expressions, the numerical evaluation of the poroelastic BIOs amounts to the evaluation of weakly-singular type integrals, for which the so-called Chebyshev-based rectangular-polar method proposed in~\cite{BG20} is applicable, and the evaluation of three-dimensional tangential derivatives can be implemented via FFT~\cite{BY20}.

The third issue is related to the spectral regularization or preconditioning. In addition to algebraic preconditioning approaches, such as sparse approximate inverse~\cite{BT98,CDGS05} and multigrid methods~\cite{HXZ14}, the analytical preconditioning approach based on the Calder\'on relation, which in fact utilizes the compositions $NS, SN$ of single-layer operator $S$ and hyper-singular operator $N$, has been discussed for solving wave scattering problems by closed surfaces~\cite{BET12,BY20,CN02} or open surfaces~\cite{ABD05,BL121,BY20}. Due to the weakly-singularity of double-layer operator, it follows easily that the acoustic Calder\'on relation can be viewed as a compact perturbation of an identity operator for the smooth closed-surface case. But this does not hold trivially in elastic case, and also in poroelastic case (which can be understood naturally since the elastic single kernels are involved in the poroelastic BIOs), on account of the fact that the classical elastic double-layer operators are not compact. It has been proved in \cite{BXY192} for two-dimensional case and in \cite{BY20} for three-dimensional case that the elastic double-layer operators $K, K'$ are polynomially compact and as a result, the values of the finite accumulation points of the eigenvalues of $NS, SN$ only depend on the Lam\'e parameters of the elastic medium. The two-dimensional poroelastic case has been discussed in \cite{ZXY21}, whereas the result does not directly fit for the three-dimensional context. It is shown in this work (see Theorem~\ref{spectra}) that the three-dimensional poroelastic double-layer operators are compact in the sense of a third-order polynomial and interestingly, the corresponding accumulation points of the eigenvalues are independent of the poroelastic parameters (see Table~\ref{Tablemodel}) except the two Lam\'e parameters. On a basis of the spectral properties of the poroelastic BIOs and analogous to the two-dimensional approach~\cite{ZXY21}, we propose two regularized CBIEs for which the eigenvalues of the combined integral operators are bounded away from zero and infinity and then it leads to significant reductions in the number of GMRES iterations required for convergence to a given residual tolerance over the original CBIEs.

This paper is organized as follows. The dynamic poroelastic scattering problem is introduced in Section~\ref{sec:2} and then in Sections~\ref{sec:3.1}-\ref{sec:3.2}, we present both the direct and indirect methods to derive the classical CBIEs, respectively. Section~\ref{sec:3.3} is arranged to give a theoretical investigation of the spectral properties of the poroelastic integral operators and the corresponding Calder\'on relation. In Section~\ref{sec:4}, regularized expressions of the strongly-singular and hyper-singular operators are presented and then the jump relations between the layer-potentials and the BIOs are proved. Section~\ref{sec:5.1} proposes two new RBIEs based on the Calder\'on relation and Section~\ref{sec:5.2} briefly describes the numerical discretization method for poroelastic BIOs. Some numerical examples are presented in Section~\ref{sec:6} to demonstrate the accuracy and efficiency of the proposed method.

\section{Poroelastic problem}
\label{sec:2}

Let $\Omega$ be a bounded domain in $\R^3$ with smooth boundary $\Gamma : = \partial \Omega $, and its exterior complement is denoted by ${\Omega ^c} = \R^3\backslash \overline \Omega$. This work is devoted to studying the numerical solutions of the three-dimensional time-harmonic problems of wave propagation in the domain $\Omega^c$ which is occupied by a linear isotropic poroelastic medium characterized by the physical parameters listed in Table~\ref{Tablemodel}.

\begin{table}[htb]
	\caption{The material parameters in poroelasticity.}
	\centering
	\begin{tabular}{|c|c|}
		\hline
		Notation & Physical meaning \\
		\hline
		$\lambda,\mu (\mu>0,3\lambda+2\mu>0)$ & Lam\'e parameters \\
		$\nu_p$ & Poisson ratio \\
		$\nu_u$ & undrained Poisson ratio \\
		$B$ & Skempton porepressure coefficient \\
		$\rho_s$ & solid density \\
		$\rho_f$ & fluid density \\
		$\rho_a=C\phi\rho_f$ & apparent mass density \\
		$\phi$ & porosity \\
		$\kappa$ & permeability coefficient \\
		$\rho=(1-\phi)\rho_s + \phi\rho_f$ & bulk density \\
		$\alpha=\frac{3(\nu_u-\nu_p)}{B(1-2\nu_p)(1+\nu_u)}$ &  compressibility \\
		$R=\frac{2\phi^2\mu B^2(1-2\nu_p)(1+\nu_u)^2}{9(\nu_u-\nu_p)(1-2\nu_u)}$ & constitutive coefficient \\
		\hline
	\end{tabular}
	\label{Tablemodel}
\end{table}

Following the Biot's theory~\cite{B41,B561,B562}, the solid displacements $u=(u_1,u_2,u_3)^\top\in H_{loc}^1(\Omega^c)^3$ and the pore pressure $p\in H_{loc}^1(\Omega^c)$ characterizing the wave propagation in poroelastic medium can be modeled by the following coupled partial differential equations
\begin{equation}
\label{model}
\begin{split}
& \Delta^*u + (\rho-\beta\rho _f)\omega ^2u - (\alpha-\beta)\nabla p = 0\\
&\Delta p + qp + i\omega\gamma\nabla\cdot u = 0
\end{split}\quad \mbox{in}\quad  \Omega^c,
\end{equation}
or equivalently, in a matrix form
\ben
LU=0, \quad L=\begin{bmatrix}
	\Delta^*  + (\rho-\beta\rho _f)\omega ^2I   & -(\alpha -\beta )\nabla \\
	i\omega\gamma\nabla \cdot & \Delta  + q
\end{bmatrix}, \quad U=(u_1,u_2,u_3,p)^\top,
\enn
where
\ben
\beta=\frac{{\omega {\phi^2}{\rho_f}\kappa}}{{i{\phi ^2} + \omega \kappa ({\rho _a} + \phi {\rho _f})}},\quad q = \frac{\omega^2\phi^2\rho_f}{\beta R},\quad
\gamma =  - \frac{{i\omega {\rho _f}(\alpha  - \beta )}}{\beta },
\enn
are abbreviations defined to simplify the representation of the problem, $\omega$ denotes the frequency, $I$ is the identity operator and $\Delta^*$ denotes the Lam\'e operator given by
\ben
\Delta^*:= \nabla \cdot \widetilde \sigma (u),
\enn
with
\ben
\widetilde \sigma (u)=\lambda (\nabla \cdot u)I+2\mu \widetilde \varepsilon (u) \quad  \mbox{and} \quad \widetilde \varepsilon (u) = \frac{1}{2}(\nabla u + (\nabla u)^\top).
\enn
Given some data $F\in H^{-1/2}(\Gamma)$, the Neumann boundary condition
\be
\label{boundary condtion}
\widetilde{T}(\partial ,\nu )U: = \begin{bmatrix}
	T(\partial ,\nu ) & - \alpha\nu \\
	- \rho_f\omega^2\nu^\top & \partial_\nu
\end{bmatrix}U = F \quad\mbox{on}\quad\Gamma
\en
is imposed for the poroelastic problem in which the traction operator $T(\partial ,\nu )$ is defined as
\ben
T(\partial ,\nu )u: = 2\mu{\partial _\nu }u + \lambda \nu \nabla \cdot u + \mu \nu \times \nabla \times u,\quad \nu=(\nu_1,\nu_2,\nu_3)^\top,
\enn
where $\nu$ denotes the outward unit normal to the boundary $\Gamma$ and $\pa_\nu:=\nu\cdot\nabla$ is the normal derivative.

It follows ~\cite{BXY191,HS20,S012} that the solution $U$ of (\ref{model}) admits a representation of the form
\ben
U=(u,p)^\top=(u^1,p^1)^\top+(u^2,p^2)^\top+(u^s,p^s)^\top
\enn
where $(u^k,p^k)$, $k=1,2,s$, satisfy
\be
\label{dec}
\begin{array}{*{20}{r}}
	(\Delta+k_1^2)u^1=0,&(\Delta+k_2^2)u^2=0,&(\Delta+k_s^2)u^s=0,\\
	{\rm{curl}}\,u^1=0,&{\rm{curl}}\,u^2=0,&{\rm{div}}\,u^s=0,\\
	(\Delta+k_1^2)p^1=0,&(\Delta+k_2^2)p^2=0,&p^s=0.
\end{array}
\en
Here, we denote by $k_p$ and $k_s$ the compressional and shear wave numbers, respectively, and they are given by
\ben
k_p:=\omega\sqrt{\frac{\rho-\beta\rho_f}{\lambda  + 2\mu}}, \quad k_s:=\omega\sqrt{\frac{\rho-\beta\rho_f}{\mu}}.
\enn
The wave numbers $k_1$, $k_2$ in (\ref{dec}), which represent the wave numbers of the fast compressional wave and slow compressional wave in poroelastic medium, respectively, are determined through
\ben
k_{1}^2+k_{2}^{2}=q(1+ \epsilon )+k_p^{2}, \quad k_1^{2}k_2^{2}=qk_p^{2}, \quad \mbox{Im}(k_i)\ge 0, i=1,2,
\enn
with $\epsilon=\frac{i\omega\gamma(\alpha-\beta)}{q(\lambda+2\mu)}$. In particular,
\ben
k_1 &=& \sqrt {\frac{1}{2}\left\{ {k_p^2 + q(1 + \varepsilon ) + \sqrt {[ {k_p^2 + q(1 + \varepsilon )}] - 4q k_p^2} } \right\}}, \\
k_2 &=& \sqrt {\frac{1}{2}\left\{ {k_p^2 + q(1 + \varepsilon ) - \sqrt {[ {k_p^2 + q(1 + \varepsilon )} ] - 4qk_p^2} } \right\}}.
\enn
To complete the statement of the poroelastic problem, we assume that the solution $U$ satisfies the following Kupradze radiation conditions as $r = \left| x \right| \to \infty $ for $l=1,2,3$ and $j=1,2,$
\ben
u^j&=&o(r^{-1}),\quad \partial_{x_l}u^j=O(r^{-2}),\\
p^j&=&o(r^{-1}),\quad \partial_{x_l}p^j=O(r^{-2}),\\
u^s&=&o(r^{-1}),\quad r(\partial_ru^s-ik_su^s)=O(r^{-1}).
\enn

\begin{remark}
For the poroelastic problem, the degrees of freedom can be determined in different ways~\cite{S09}. Compared with the formulation in terms of the solid displacement and the fluid displacement, and the formulation in terms of the solid displacement and the seepage displacement, the above model enjoys the lowest number of unknowns. For the uniqueness analysis of the dynamic poroelastic problem, we refer to~\cite{DS63}.
\end{remark}

\section{Boundary integral equations}
\label{sec:3}
In this section, we introduce the hyper-singular BIEs for solving the poroelastic problem together with some theoretical study of the properties of BIOs. Based on the Green's identities and potential theory, direct and indirect boundary integral formulations are derived, respectively. We begin with the first and second Green's identities for the poroelastic problems in $\Omega$ (analogous to the problem in $\Omega^c$). For $U=(u^\top,p)^\top$ and $V=(v^\top,\theta)^\top$, the first Green's identity reads
\be
\label{FGI}
\int_{\Omega} LU\cdot V dx + A_\Omega(U,V)=\int_{{\Omega }} {\widetilde T(\partial ,\nu )} U\cdot Vds,
\en
while the second Green's identity admits
\be
\label{SGI}
\int_{\Omega} \left(LU\cdot V-U\cdot L^*V\right)\,dx= \int_\Gamma \left( \widetilde{T}(\pa,\nu)U\cdot V-U\cdot \widetilde{T}^*(\pa,\nu)V\right)\,ds.
\en
Here, $A_\Omega$ denotes a bilinear form defined by
\ben
A_\Omega(U,V)&:=&\int_\Omega ( \widetilde \sigma (u):\widetilde \varepsilon (v) + (\rho  - \beta {\rho _f}){\omega ^2}u \cdot v - \alpha p\nabla  \cdot v \\
	    &&+ \beta \nabla p \cdot v- {\rho _f}{\omega ^2}u \cdot \nabla \theta  + \frac{{{\rho _f}{\omega ^2}\alpha }}{\beta }\nabla  \cdot u\theta  + \nabla p \cdot \nabla \theta  + qp\theta  ) dx,
\enn
where $L^*$ denotes the adjoint operator of $L$ defined by
\ben
L^*=\begin{bmatrix}
	\Delta^*  + (\rho-\beta\rho _f)\omega ^2I   & -i\omega\gamma\nabla  \\
	(\alpha -\beta )\nabla\cdot& \Delta  + q
\end{bmatrix},
\enn
and $\widetilde{T}^*(\pa,\nu)$ is the boundary operator given by
\be
\widetilde T^ * (\partial ,\nu ) = \begin{bmatrix}
	{T(\partial ,\nu )}& -\frac{ \rho_f\omega^2\alpha } {\beta}\nu\\
	{ - \beta {\nu ^\top}}& \partial _{\nu}
\end{bmatrix}.
\label{Tadjoint}
\en
It is known~\cite{CD95,S012} that the fundamental solution of the operator $L^*$ in $\R^3$ is given by
\ben
E(x,y)=
\begin{bmatrix}
	E_{11}(x,y) & E_{12}(x,y) \\
	E_{21}^\top(x,y) & E_{22}(x,y)
\end{bmatrix},\quad x\ne y,
\enn
with
\ben
&& {E_{11}}(x,y) = \frac{1}{\mu}{\gamma _{{k_s}}}(x,y)I + \frac{1}{{(\rho  - \beta {\rho _f})}{\omega ^2}}{\nabla _x}\nabla _x^\top\begin{pmatrix} {\gamma _{{k_s}}}(x,y) - {\frac{k_{p}^{2}-k_{2}^{2}}{k_{1}^{2}-k_2^2}}{\gamma _{{k_1}}}(x,y) \\
+ {\frac{k_{p}^{2}-k_{1}^{2}}{k_{1}^{2}-k_{2}^{2}}}{\gamma _{{k_2}}}(x,y) \end{pmatrix},\\
&& E_{12}(x,y) =  \frac{i\omega\gamma}{(\lambda+2\mu)(k_1^2-k_2^2)}\nabla _x[ \gamma _{k_1}(x,y) - \gamma _{k_2}(x,y) ],\\
&& E_{21}(x,y) =  -\frac{\alpha-\beta}{(\lambda+2\mu)(k_1^2-k_2^2)}\nabla _x\left[ \gamma _{k_1}(x,y) - \gamma _{k_2}(x,y) \right],\\
&& E_{22}(x,y) = -\frac{1}{(k_1^2 - k_2^2)}\left[ (k_p^2 - k_1 ^2)\gamma _{k_1}(x,y) - (k_p^2 - k_2 ^2)\gamma _{k_2}(x,y) \right],
\enn
in which
\ben
\gamma_{k_t}(x,y)=\frac{\mbox{exp}(ik_t\left|x-y\right|)}{4\pi\left|x-y\right|}, \quad x \ne y, \quad t=s,p,1,2,
\enn
is the fundamental solution of the Helmholtz equation in $\R^3$ with wave number $k_t$.

\subsection{Direct method}
\label{sec:3.1}
It follows from the Green's formulas (\ref{SGI}) that the solution of (\ref{model}) can be represented in the form
\be
\label{dr1}
U(x)=\mathcal{D}(U)(x)-\mathcal{S}(\widetilde T(\partial,\nu)U)(x), \quad x\in \Omega^c,
\en
where $\mathcal{S}$ and $\mathcal{D}$ are the single-layer and double-layer potentials given by
\be
\label{single}
\mathcal{S}(\varphi)(x)&:=&\int_\Gamma  (E(x,y))^\top\varphi (y)ds_y,\quad x\notin\Gamma,\\
\label{double}
\mathcal{D}(\varphi)(x)&:=&\int_\Gamma (\widetilde T^*(\pa_y,\nu_y)E(x,y))^\top \varphi(y)ds_y,\quad x\notin\Gamma,
\en
respectively. Introduce the BIOs for the poroelasticity in the sense of principle value or Hadamard finite part as follows
\be
\label{SO}
S(\varphi)(x)&:=&\int_\Gamma  (E(x,y))^\top\varphi (y)ds_y,\quad x\in\Gamma,\\
\label{DO}
K(\varphi)(x)&:=&\int_\Gamma (\widetilde T^*(\pa_y,\nu_y)E(x,y))^\top \varphi(y)ds_y,\quad x\in\Gamma,\\
\label{TDO}
K'(\varphi)(x) &:=&\widetilde T(\partial_x,\nu_x) \int_\Gamma  (E(x,y))^\top\varphi (y)ds_y,\quad x\in\Gamma,\\
\label{HO}
N(\varphi)(x) &:=& \widetilde T(\partial_x,\nu_x)\int_\Gamma  \left(\widetilde T^*(\partial_y,\nu_y)E(x,y)\right)^\top\varphi (y)ds_y,\quad x\in\Gamma,
\en
where $S$, $K$, $K'$ and $N$ are called, respectively, the single-layer, double-layer, transpose of double-layer, and hyper-singular BIOs. Then we conclude the jump relation results associated with the poroelastic layer potentials and BIOs in the following theorem.

\begin{theorem}
\label{jump}
For $x\in\Gamma$, the following jump relations hold:
\ben
&&\lim_{h\rightarrow 0^+,z=x\pm h\nu_x} \mathcal{S}(\varphi)(z)=S(\varphi)(x), \\
&&\lim_{h\rightarrow 0^+,z=x\pm h\nu_x} \mathcal{D}(\varphi)(z)= \left(\pm\frac{1}{2}I+K\right)(\varphi)(x),\\
&&\lim_{h\rightarrow 0^+,z=x\pm h\nu_x} \widetilde T(\partial_z,\nu_x)\mathcal{S}(\varphi)(z)= \left(\mp\frac{1}{2}I+K'\right)(\varphi)(x), \\
&&\lim_{h\rightarrow 0^+,z=x\pm h\nu_x} \widetilde T(\partial_z,\nu_x)\mathcal{D}(\varphi)(z)= N(\varphi)(x).
\enn
\end{theorem}
\begin{remark}
The proof of this theorem relies on the study of the regularized expressions of the integral operators that will be presented in Section~\ref{sec:4} and thus, will be reported after that.
\end{remark}

Now applying the jump conditions, we are led to the BIEs on $\Gamma$
\be
\label{dBIEs1}
U(x)=\left(\frac{1}{2}I+K\right)(U)(x)-S(\widetilde T(\partial,\nu)U)(x), \quad x\in \Gamma,
\en
and
\be
\label{dBIEs2}
\widetilde T(\partial_x,\nu_x)(U)(x)=N(U)(x)+\left(\frac{1}{2}I-K'\right)(\widetilde T(\partial,\nu)U)(x),\quad  x\in \Gamma.
\en
Combining the BIEs (\ref{dBIEs1})-(\ref{dBIEs2}) results into the so-called Burton-Miller formulation \cite{BM71} on $\Gamma$
\be
\left[i\eta\left(\frac{1}{2}I-K\right)-N\right](U)(x)+\left[\frac{1}{2}I+K'+i\eta S\right](\widetilde T(\partial,\nu)U)(x)=0,
\en
where $\eta\ne0$ is a combination coefficient. Using the boundary condition (\ref{boundary condtion}), we obtain the direct combined boundary integral equation (DCBIE)
\be
\label{DCBIE}
\left[i\eta\left(\frac{1}{2}I-K\right)-N\right](U)(x)=-\left[\frac{1}{2}I+K'+i\eta S\right](F)(x),\quad x\in \Gamma.
\en

\subsection{Indirect method}
\label{sec:3.2}

The indirect boundary integral formulations can also be used for solving the poroelastic problems, which also allow for a suitable tool to test each operator separately. From the potential theory, the unknown function $U$ of (\ref{model}) can be represented by a combination of the single-layer and double-layer potentials
\be
U(x)=(\mathcal{D}-i\eta \mathcal{S})(\varphi)(x), \qquad x\in \Omega^c, \quad \eta \ne 0.
\label{CSDLP}
\en
Operating with the boundary operator $\widetilde T(\pa,\nu)$ on $(\ref{CSDLP})$, taking the limit as in Theorem~\ref{jump} and applying the boundary condition (\ref{boundary condtion}), we can obtain the indirect combined boundary integral equation (ICBIE)
\be
\label{ICBIE}
\left[i\eta\left( \frac{I}{2}-K^\prime \right)+N\right](\varphi)(x)&=&F, \quad x\in \Gamma.
\en


\subsection{Operator properties}
\label{sec:3.3}

Assuming that the boundary $\Gamma$ is sufficiently smooth, the BIOs are continuous mappings between the following spaces~\cite{HW08,YHX17}
\be
S\quad &:& \quad (H^{-1/2}(\Gamma))^4\to (H^{1/2}(\Gamma))^4,\\
K,K^\prime \quad&:& \quad (H^{\pm 1/2}(\Gamma))^4\to (H^{\pm 1/2}(\Gamma))^4,\\
N \quad&:&\quad (H^{1/2}(\Gamma))^4\to (H^{-1/2}(\Gamma))^4.
\en
and the following Calder\'{o}n relations hold:
\be
SN&=& K^2-\frac{1}{4}I, \quad NS={K^\prime}^2-\frac{1}{4}I,\\
KS&=&SK^\prime,\quad NK=K^\prime N.
\en

As analytical preconditioning techniques, the Calder\'{o}n relations have been investigated and utilized in regularized BIE methods~\cite{BET12,BY20,ZXY21}, which require the spectral study of the BIOs, to construct BIE systems possessing highly favorable spectral properties. The main reason is that the eigenvalues of the hyper-singular integral operator $N$ accumulate at infinity. As a result, obtaining the solutions of some integral equations, for example (\ref{DCBIE}) and (\ref{ICBIE}) in this work, by means of Krylov-subspace iterative solvers such as GMRES generally requires large numbers of iterations. To overcome this difficulty, the spectral properties of the integral operators $K,K'$ and the associated Calder\'{o}n relations $NS$ for two-dimensional poroelasticity are investigated in \cite{ZXY21} and then a regularized BIE method is proposed. However, as proved in the following theorem, the three-dimensional integral operators $K$ and $K'$ enjoy spectral properties different from those in the two-dimension case.

\begin{theorem}
\label{spectra}
Let $\widetilde{I}_{\lambda,\mu}$ be a matrixed operator given by
\ben
\widetilde{I}_{\lambda,\mu}=\begin{bmatrix}
	C_{\lambda ,\mu }^2I & 0\\
	0 & 0
	\end{bmatrix},
\enn
where $C_{\lambda,\mu}$ is a constant satisfying
\ben
	0<C_{\lambda ,\mu }= \frac{\mu }{2(\lambda  + 2\mu )} < \frac{3}{8}.
\enn
Then $K^\prime({K^\prime}^2-\widetilde{I}_{\lambda,\mu}): (H^{1/2}(\Gamma))^4\to (H^{1/2}(\Gamma))^4$ is compact. Furthermore, the spectrum of $K^{\prime}$ consists of three non-empty sequences of eigenvalues which converge to $0$, $C_{\lambda,\mu}$ and $-C_{\lambda,\mu}$ respectively.
\end{theorem}
\begin{proof}
Analogous to the proof of \cite[Theorem 3.1]{ZXY21}, it is sufficient to consider the static $(\omega=0)$ BIO corresponding to $K^\prime$ which can be formulated as
\ben
	K_0'(U)(x) = \widetilde T_0(\partial _x,\nu _x)\int_\Gamma (E_0(x,y))^\top U(y)ds_y=\begin{bmatrix}
		K'_{1,0}& K'_{2,0}\\
		K'_{3,0}& K'_{4,0}
	\end{bmatrix}\begin{bmatrix}
			u\\
			p
	\end{bmatrix}(x),
\enn
where
\ben
	\widetilde T_0(\partial _x,\nu _x)=\begin{bmatrix}
		T(\partial _x,\nu _x)&-\alpha\nu_x\\
		0& \partial_{\nu_x}
	\end{bmatrix},
\enn
and
\ben
	E_{0}(x,y)=\begin{bmatrix}
		E_{0,11}&E_{0,12}\\
		E_{0,21}^\top&E_{0,22}
	\end{bmatrix}=\begin{bmatrix}
		E_{e,0}&0 \\
		\frac{\alpha(x-y)^\top}{8\pi(\lambda+2\mu)\left|x-y\right|}& \frac{1}{4\pi\left|x-y\right|}
	\end{bmatrix}
\enn
is the fundamental solution of static poroelastic problem with
	\begin{equation*}
	E_{e,0}(x,y)=\frac{\lambda+3\mu}{8\pi\mu(\lambda+2\mu)}\left\{ \frac{1}{\left|x-y\right|}I+\frac{\lambda+\mu}{(\lambda+3\mu)}\frac{(x-y)(x-y)^\top}{\left|x-y\right|^3} \right\}
	\end{equation*}
being the fundamental solution of Lam\'e equation. It can be verified that $K_{3,0}'=0$ and the kernels of $K'_{j,0},j=2,4$ admit weak singularity implying that $K'_{j,0},j=2,4$ are compact. From \cite{AKM}, we know that $K_{1,0}^\prime({K^\prime}^2_{1,0}-C_{\lambda ,\mu }^2I): (H^{1/2}(\Gamma))^4\to (H^{1/2}(\Gamma))^4$ is compact. Therefore,
	\ben
	&&K_0^\prime({K_0^\prime}^2-\widetilde{I}_{\lambda,\mu})\\
    &=&\begin{bmatrix}
	K_{1,0}^\prime({K^\prime}^2_{1,0}-C_{\lambda ,\mu }^2I)&{K^\prime}^2_{1,0}K^\prime_{2,0}+K^\prime_{1,0}K^\prime_{2,0}K^\prime_{4,0}+K^\prime_{2,0}{K^\prime}^2_{4,0}\\
	0&{K^\prime}^3_{4,0}
	\end{bmatrix}
	\enn
	is compact. A direct calculation yields
	\ben
	&&K^\prime({K^\prime}^2-\widetilde{I}_{\lambda,\mu})=K_0^\prime({K_0^\prime}^2-\widetilde{I}_{\lambda,\mu})+K_c,\\
    &&K_c= K_0^\prime(K^\prime-K^\prime_0)K^\prime+ {K_0^\prime}^2(K^\prime-K^\prime_0)+(K^\prime-K^\prime_0)({K_0^\prime}^2-\widetilde{I}_{\lambda,\mu}).
	\enn
The compactness of $K^\prime-K^\prime_0$ indicates the compactness of $K_c$ and then further implies that $K^\prime({K^\prime}^2-\widetilde{I}_{\lambda,\mu})$ is compact. The proof is completed.  	
\end{proof}

The spectral property of the operator $K$ is similar to $K'$. Relying on the Calder\'{o}n relations~\cite{HW08}
\be
NS={K^\prime}^2-\frac{1}{4}I,\quad SN=K^2-\frac{1}{4}I,
\en
we can conclude that the spectrum of bith the composite operators $NS$ and $SN$ consist of two non-empty sequences of eigenvalues which converge to $- \frac{1}{4}$ and $-\frac{1}{4}+C_{\lambda,\mu}^2$, respectively.

To verify the above results numerically, we consider the problem of poroelastic scattering by a unit ball, and we choose the same values of parameters as in Section~\ref{sec:6}. Consequently, the constant $C_{\lambda,\mu}=0.1875$. Utilizing the following re-expressions of the BIOs together with the discretization method presented in Sections~\ref{sec:4}-\ref{sec:5}, the eigenvalue distributions of the integral operators $K'$, $K$, $NS$ and $SN$ are displayed in Figure~\ref{KNSeig} showing an agreement with our theoretical results. Based on these results, the corresponding regularized BIE method is proposed in Section~\ref{sec:5} for solving the three-dimensional poroelastic problem.

\begin{figure}[htb]
	\centering
	\begin{tabular}{cc}
		\includegraphics[scale=0.25]{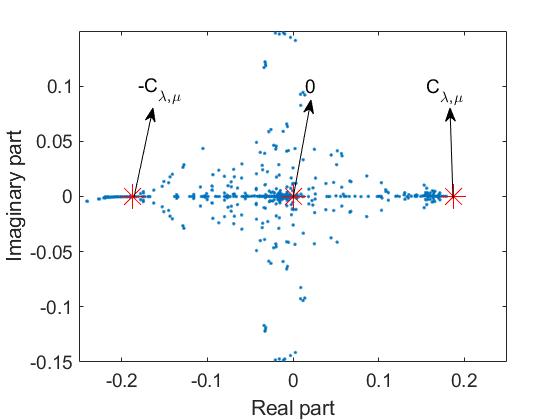} &
		\includegraphics[scale=0.25]{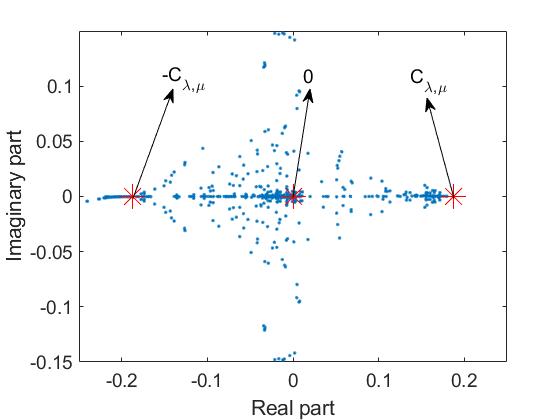} \\
		(a) $K'$ & (b) $K$  \\
		\includegraphics[scale=0.25]{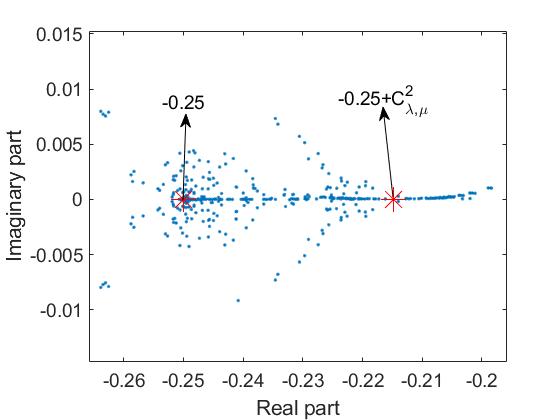} &
		\includegraphics[scale=0.25]{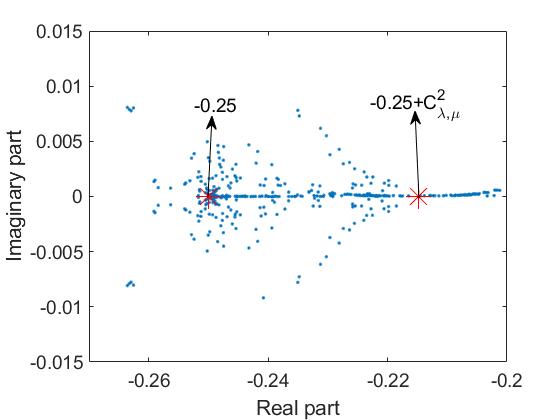}  \\
		(c) $NS$ & (d) $SN$  \\
	\end{tabular}
	\caption{Eigenvalue distributions of the operators $K'$ (a), $K$ (b), $NS$ (c) and $SN$ (d) for a unit ball scatterer.}
	\label{KNSeig}
\end{figure}

\section{Strong-singularity and hyper-singularity regularization}
\label{sec:4}

As aforementioned, the integral operators $K$ (as well as $K'$) and $N$ are strongly-singular and hyper-singular, respectively. In this section, we will present new (regularized) expressions for these operators. More precisely, with the help of the tangential G\"{u}nter derivative~\cite{BXY191}, these operators will be expressed in terms of compositions of weakly-singular integral operators ($K_j^i$, $K_j^{\prime i}$, $N_j^i$ appearing in the following Theorems (\ref{regK})-(\ref{regN4})) and tangential-derivative operators ({\it The formulations to derive these regularized expressions will be shown in the supplemented material.}). Employing these formulations together with the rectangular-polar quadrature method proposed in~\cite{BG20,BY20} and the iteration solver GMRES then leads to our boundary integral solver for the poroelastic problem.

We begin with the G\"{u}nter derivative operator $M(\partial ,\nu )$ defined as
\ben
M(\partial,\nu)u(x)=\partial_\nu u-\nu(\nabla  \cdot  u)+\nu  \times {\rm{curl}}\, u.
\enn
From~\cite{BY20}, we know that
\ben
M(\partial ,\upsilon ) = \begin{pmatrix}
	0&-\widetilde \partial_3^S&\widetilde \partial_2^S\\
	\widetilde \partial_3^S&0&-\widetilde \partial_1^S\\
	-\widetilde \partial_2^S&\widetilde \partial_1^S&0
	\end{pmatrix},
\enn
where $\widetilde \partial_i^S$, $i=1,2,3$ are the components of $\nu\times\nabla^S=(\partial_1^S,\partial_2^S,\partial_3^S)^\top$, in which $\nabla^S$ is the surface gradient defined as
$\nabla^S=\nabla u-\nu\partial_\nu u$. Then the traction operator $T(\partial ,\nu )$ can be rewritten as
\ben
T(\partial ,\nu )u(x) = (\lambda  + 2\mu )\nu (\nabla  \cdot u) + \mu {\partial _\nu }u + \mu M(\partial ,\nu )u.
\enn

\subsection{Strong-singularity regularization}
\label{sec:4.1}
We first consider the operators $K$ and $K'$ in the forms of
\ben
K(U)(x)= \begin{bmatrix}
	K_1& K_2\\
	K_3 & K_4
\end{bmatrix}\begin{bmatrix}
	u\\
	p
\end{bmatrix}(x),\quad K'(U)(x)= \begin{bmatrix}
	K_1'& K_2'\\
	K_3' & K_4'
\end{bmatrix}\begin{bmatrix}
	u\\
	p
\end{bmatrix}(x),\quad x\in\Gamma.
\enn
Then the following regularized formulations can be obtained.

\begin{theorem}
	\label{regK}
	The operators $K_j,j=1,\cdots,4$ can be expressed as
	\be
	\label{KP}
	K_j=K_j^1+K_j^2M(\partial,\nu)+\left\{ {K_j^3M(\partial,\nu)} \right\}^\top,\quad j=1,\cdots,4,
	\en	
	where $K_1^3=K_2^2=K_3^3=K_4^2=K_4^3=0$, and
	\ben
	&&K_1^1(u)(x)=\int_\Gamma \left[\partial_{\nu _y}\gamma_{k_s}(x,y)I-\frac{\rho_f\omega^2\alpha}{\beta}E_{21}(x,y)\nu_y^\top\right]u(y)ds_y \\
    &&\quad -\int_\Gamma  \nabla _y\left[  (\gamma _{k_s}(x,y) - \gamma _{k_1}(x,y)) - \frac{k_2^2 - q}{k_1^2 - k_2^2}(\gamma _{k_1}(x,y) - \gamma _{k_2}(x,y))\right]\nu _y^\top u(y) ds_y, \\
	&&K_1^2(u)(x) = \int_\Gamma  \left[2\mu E_{11}(x,y) - \gamma _{k_s}(x,y)I\right] u(y)ds_y,\\
	&&K_2^1(p)(x)\\
	&&=-\int_\Gamma  \left[\frac{\alpha-\beta }{(k_1^2 - k_2^2)(\lambda+2\mu)} (k_1^2\gamma _{k_1}(x,y) - k_2^2\gamma _{k_2}(x,y)) + \beta E_{11}(x,y)\right]\nu _yp(y)ds_y,\\
	&&K_2^3(p)(x) = \frac{\alpha-\beta }{(\lambda+2\mu)(k_1^2-k_2^2)}\int_\Gamma \nabla _y^\top\left[\gamma_{k_1}(x,y) - \gamma_{k_2}(x,y)\right]p(y)ds_y,\\
	&&K_3^1(u)(x) \\
	&&= \int_\Gamma  \left[  \frac{i\omega\gamma}{(k_1^2-k_2^2)}(k_1^2\gamma _{k_1}(x,y)-k_2^2\gamma_{k_2}(x,y))-\frac{\rho_f\omega^2\alpha}{\beta}E_{22}\right]\nu _y^\top u(y)ds_y,\\
	&&K_3^2(u)(x)=-\frac{2i\mu\omega\gamma}{(k_1^2-k_2^2)(\lambda+2\mu)}\int_\Gamma \nabla_y^\top\left[\gamma _{k_1}(x,y) - \gamma_{k_2}(x,y)\right]u(y) ds_y,\\
	&&K_4^1(p)(x)=\frac{i\omega\gamma\beta}{(\lambda+2\mu)(k_1^2-k_2^2)}\int_\Gamma \partial_{\nu_y}\left[\gamma _{k_1}(x,y) - \gamma_{k_2}(x,y)\right]p(y)ds_y\\  &&\quad-\frac{1}{k_1^2-k_2^2}\int_\Gamma\partial_{\nu_y}\left[(k_p^2-k_1^2)\gamma_{k_{1}}(x,y)-(k_p^2-k_2^2) \gamma_{k_{2}}(x,y)\right]p(y)ds_y.	
\enn
\end{theorem}

\begin{theorem}
	\label{regKP}
	The operators $K'_j,j=1,\cdots,4$ can be expressed as
	\be
	\label{KPP}
	K'_j=K_j^{\prime1}+M(\partial_x,\nu_x)K_j^{\prime2}+M(\partial_x,\nu_x):K_j^{\prime3},\quad j=1,\cdots,4,
	\en
	where $K_1^{\prime3}=K_2^{\prime3}=K_3^{\prime2}=K_4^{\prime2}=K_4^{\prime3}=0$, and
	\ben
	&&K_1^{\prime1}(u)(x)=\int_\Gamma  \left[\partial_{\nu_x}\gamma_{k_s}(x,y)I - \alpha \nu_xE_{12}^\top(x,y) \right] u(y)ds_y\\
    &&\quad-\int_\Gamma  \nu _x\nabla _x^\top\left[  (\gamma _{k_s}(x,y) - \gamma _{k_1}(x,y))-  \frac{k_2^2 - q}{k_1^2 - k_2^2}(\gamma _{k_1}(x,y) - \gamma _{k_2}(x,y))\right]u(y) ds_y, \\
	&&K_1^{\prime2}(u)(x) = \int_\Gamma  \left[2\mu E_{11}(x,y) - \gamma _{k_s}(x,y)I\right] u(y)ds_y,\\
	&&K_2^{\prime1}(p)(x) \\
	&&= \int_\Gamma  \left[\frac{\alpha-\beta }{k_1^2 - k_2^2} (k_1^2\gamma _{k_1}(x,y) - k_2^2\gamma _{k_2}(x,y))\nu _x - \alpha E_{22}(x,y)\nu _x\right]p(y)ds_y,\\
	&&K_2^{\prime2}(p)(x) = -\frac{2\mu(\alpha-\beta )}{(\lambda+2\mu)(k_1^2-k_2^2)}\int_\Gamma \nabla _x\left[\gamma_{k_1}(x,y) - \gamma_{k_2}(x,y)\right]p(y)ds_y,\\
	&&K_3^{\prime1}(u)(x) =-\rho_f\omega^2 \int_\Gamma  \nu _x^\top E_{11}(x,y) u(y)ds_y\\
	&&\quad-\frac{i\omega\gamma}{(k_1^2-k_2^2)(\lambda+2\mu)}\int_\Gamma\left[k_1^2\gamma _{k_1}(x,y) - k_2^2\gamma_{k_2}(x,y)\right]\nu _x^\top u(y)ds_y,\\
	&&K_3^{\prime3}(u)(x)=\frac{i\omega\gamma}{(k_1^2-k_2^2)(\lambda+2\mu)}\int_\Gamma u(y)\nabla_x^\top\left[\gamma _{k_1}(x,y) - \gamma_{k_2}(x,y))\right] ds_y,\\
	&&K_4^{\prime1}(p)(x)\\
	&&=-\int_\Gamma  \nu_x^\top\left[\rho_f\omega^2 E_{21}(x,y) + \nabla _x\left(\frac{k_p^2 - k_1^2}{k_1^2 - k_2^2}\gamma_{k_1}(x,y) - \frac{k_p^2 - k_2^2}{k_1^2 - k_2^2}\gamma _{k_2}(x,y)\right)\right] p(y)ds_y.
\enn
\end{theorem}

\subsection{Hyper-singularity regularization}
\label{sec:4.2}

In this subsection, we investigate the hyper-singular operator $N$ with
\ben
N(\psi)(x)=\begin{bmatrix}
	N_1 & N_2\\
	N_3 & N_4
\end{bmatrix}\begin{bmatrix}
	u\\
	p
\end{bmatrix}(x),\quad x\in\Gamma.
\enn
The regularized formulations for the operators $N_j,j=1,\cdots,4$ are given in the following theorems.
\begin{theorem}
	\label{regN1}
	The hyper-singular operator $N_1$ can be expressed as
	\ben
	\label{N1}
	N_1=N^1_1+M(\partial,\nu)N^2_1M(\partial,\nu)+\tau_2N^3_1\tau_1+M(\partial,\nu)N^4_1+N^5_1M(\partial,\nu),
	\enn
	where
	\ben
	&&N^1_{1}(u)(x)=-(\rho-\beta\rho_f) \omega^2\int _\Gamma \gamma_{k_s}\left(x,y\right)\left(\nu_x\nu_y^\top -\nu^\top_x\nu_yI\right)u(y)ds_{y}\\
	&&\quad+\int_\Gamma\left[C_{1}\gamma_{k_{1}}(x,y)-C_{2}\gamma_{k_{2}}(x,y)\right] \nu_{x}\nu_{y}^{\top}u(y)ds_{y},\\
	&&N^2_1(u)(x)=\int_{\Gamma}\left[4\mu^2E_{11}(x,y)-3\mu\gamma_{k_s}(x,y)I\right]u(y)ds_y,\\
	&&N^3_1(u)(x)=\mu\int_\Gamma \gamma_{k_s}(x,y)u(y)ds_y,
\enn
\ben
	&&N^4_1(u)(x)=\mu\int_\Gamma\partial_{\nu_y}\gamma_{k_s}(x,y) u(y)ds_y\\
	&&\quad -\int_\Gamma \nabla _y\left[2\mu(\gamma_{k_s}(x,y)-\gamma_{k_1}(x,y) )-C_3(\gamma_{k_1}(x,y)-\gamma_{k_2}(x,y) )\right] \nu^\top_y u(y)ds_y,\\
	&&N^5_1(u)(x)=\mu\int_\Gamma\partial_{\nu_x}\gamma_{k_s}(x,y) u(y)ds_y\\
	&&\quad -\int_\Gamma \nu_x\nabla _x^\top \left[2\mu(\gamma_{k_s}(x,y)-\gamma_{k_1}(x,y) )-C_4(\gamma_{k_1}(x,y)-\gamma_{k_2}(x,y) )\right]u(y)ds_y,
\enn
	with the constants $C_{i}$, $i=1,2,3,4$ being given by
	\ben
	C_1 = \frac{\beta k_1^2(k_1^2-q)(\lambda+2\mu)-i\omega \alpha \gamma\beta k_1^2-\rho_f\omega^2\alpha(\alpha-\beta)k_1^2-\rho_f\omega^2\alpha^2(k_p^2-k_1^2)}{ \beta(k_1^2 - k_2^2)}, \\
	C_2 = \frac{\beta k_2^2(k_2^2-q)(\lambda+2\mu)-i\omega \alpha \gamma\beta k_2^2-\rho_f\omega^2\alpha(\alpha-\beta)k_2^2-\rho_f\omega^2\alpha^2(k_p^2-k_2^2)}{ \beta(k_1^2 - k_2^2)}, \\
	C_3=\frac{2\mu }{k_1^2 - k_2^2}(k_2^2-q-\frac{\rho_f\omega^2\alpha(\alpha-\beta)}{\beta(\lambda+2\mu)}),\quad C_4 = \frac{2\mu }{k_1^2 - k_2^2}(k_2^2 - q-\frac{i\omega \gamma \alpha }{\lambda  + 2\mu } ),
	\enn
	
\end{theorem}
and where, for a scalar field $p$ and a vector field $v$, the operators $\tau_1$ and $\tau_2$ are defined by
\ben
\tau_1v=\nu\times \nabla^Sp, \quad \tau_2V=(\nu\times\nabla^S)\cdot v.
\enn

\begin{theorem}
	\label{regN2}
	The hyper-singular operator $N_2$ can be expressed as
	\ben
	\label{N2}
	N_2=N^1_2+M(\partial,\nu)N^2_2M(\partial,\nu)+M(\partial,\nu)N^3_2,
	\enn
	where
	\ben
	&&N^1_2(p)(x)=\frac{\alpha-\beta }{k_1^2 - k_2^2}\int_\Gamma  \partial _{\nu _y}\left[(k_1^2\gamma _{k_1}(x,y) - k_2^2\gamma _{k_2}(x,y))\right] \nu _xp(y)ds_y\\
	&&\quad +\beta \int_\Gamma  \left[\nu _x\nabla _x^\top(\gamma _{k_s}(x,y) - \gamma _{k_1}(x,y)) - \partial _{\nu _x}\gamma _{k_s}(x,y)I\right]\nu _yp(y) ds_y\\
	&&\quad+(\frac{i\omega\gamma\alpha\beta}{(k_1^2-k_2^2)(\lambda+2\mu)}-\frac{\beta(k_2^2-q)}{k_1^2-k_2^2})\int_\Gamma \nu _x\nabla _x^\top(\gamma _{k_1}(x,y) - \gamma _{k_2}(x,y))\nu _yp(y) ds_y\\
	&&\quad+\frac{\alpha }{k_1^2 - k_2^2}\int_\Gamma  \partial _{\nu _y}\left[(k_p^2 - k_1^2)\gamma _{k_1}(x,y)-(k_p^2 - k_2^2)\gamma _{k_2}(x,y) \right]\nu _xp(y)ds_y,\\
	&&N^2_2(p)(x)=\frac{2\mu(\alpha-\beta)}{(\lambda+2\mu)(k^2_1-k^2_2)}\int_\Gamma \left\{ \nabla _y^\top\left[ \gamma_{k_1}(x,y) - \gamma _{k_2}(x,y) \right]p(y) \right\}^\top ds_y,\\
	&&N^3_2(p)(x)=- \beta \int_\Gamma \left[2\mu E_{11}(x,y)- \gamma _{k_s}(x,y)I)\right]\nu _yp(y)ds_y\\
	&&\quad-\frac{2\mu(\alpha-\beta)}{(\lambda+2\mu)(k^2_1-k^2_2)}\int_\Gamma \left[k_1^2\gamma_{k_1}(x,y)-k_2^2\gamma_{k_2}(x,y)\right]\nu_y p(y)ds_y .	
\enn
\end{theorem}
\begin{theorem}
	\label{regN3}
	The hyper-singular operator $N_3$ can be expressed as
	\ben
	N_3=N^1_3+N^2_3M(\partial,\nu)+M(\partial,\nu):N^3_3M(\partial,\nu),
	\label{N3}
	\enn
	where
	\ben
	&& N^1_3(u)(x)=\frac{i\omega\gamma}{k_1^2-k_2^2}\int_\Gamma \partial_{\nu_x}\left[k_1^2\gamma _{k_1}(x,y) - k_2^2\gamma _{k_2}(x,y)\right]\nu_y^\top u(y)ds_y\\
	&&\quad-\rho_f\omega^2 \int_\Gamma \left[\partial_{\nu_x}\left(\gamma _{k_s}(x,y) - \gamma _{k_1}(x,y)\right)\nu_y^\top+\partial_{\nu _y}\gamma_{k_s}(x,y)\nu_x^\top \right]u(y)ds_y\\
	&&\quad+(\frac{\rho_f\omega^2(k_2^2-q)}{k_1^2-k_2^2}-\frac{\rho_f^2\omega^4\alpha(\alpha-\beta)}{\beta(\lambda+2\mu)(k_1^2-k_2^2)})\int_\Gamma \partial_{\nu_x}\left[\gamma _{k_1}(x,y) - \gamma _{k_2}(x,y)\right]\nu_y^\top u(y)ds_y\\
	&&\quad +\frac{\rho_f\omega^2\alpha}{\beta(k_1^2-k_2^2)}\int_\Gamma \partial_{\nu_x}\left[(k_p^2-k_1^2)\gamma _{k_1}(x,y) - (k_p^2-k_2^2)\gamma _{k_2}(x,y)\right]\nu_y^\top u(y)ds_y,\\
	&&N^2_3(u)(x)=-\rho_f\omega^2\int_\Gamma  \nu_x^\top\left[2\mu E_{11}(x,y) - \gamma _{k_s}(x,y)I\right]u(y) ds_y\\
	&&\quad - \frac{2i\mu \omega \gamma }{(\lambda  + 2\mu )(k_1^2 - k_2^2)}\int_\Gamma  \left[k_1^2\gamma _{k_1}(x,y) - k_2^2\gamma _{k_2}(x,y)\right]\nu _x^\top u(y) d{s_y},\\
	&&N^3_3(u)(x)=\frac{2i\mu \omega \gamma }{(\lambda+2\mu )(k_1^2 - k_2^2)} \int_\Gamma  u(y)\nabla_x^\top\left[\gamma_{k_1}(x,y)-\gamma_{k_2}(x,y)\right] ds_y.
\enn
\end{theorem}

\begin{theorem}
	\label{regN4}
	The hyper-singular operator $N_4$ can be expressed as
	\ben
	N_4=N^1_4+\tau_2N^2_4\tau_1,
	\label{N4}
	\enn
	where
	\ben
	&&N^1_4(p)(x) = \rho_f\omega^2\beta\int_\Gamma  \nu _x^\top E_{11}(x,y)\nu_yp(y) ds_y\\
	&&\quad+\frac{i\omega\gamma\beta+\rho_f\omega^2(\alpha-\beta)}{(\lambda+2\mu)(k^2_1-k^2_2)}\int_\Gamma  \left[k^2_1\gamma _{k_1}(x,y) - k^2_2\gamma _{k_2}(x,y)\right] \nu_x^\top\nu_y p(y)ds_y\\
    &&\quad -\frac{1}{k^2_1-k^2_2}\int_\Gamma  \left[(k^2_p-k^2_1)k^2_1\gamma_{k_1}(x,y)-(k^2_p-k^2_2)k^2_2\gamma_{k_2}(x,y)\right]\nu_x^\top\nu_y p(y)ds_y,\\
    &&N^2_4(p)(x)=\frac{\rho_f\omega^2(\alpha-\beta)+i\omega\beta\gamma}{(\lambda+2\mu)(k_1^2-k_2^2)} \int_\Gamma\left[\gamma_{k_1}(x,y)-\gamma_{k_2}(x,y)\right]p(y)ds_y \\
    &&\quad -\frac{1}{k^2_1-k^2_2}\int_\Gamma \left[(k^2_p-k^2_1)\gamma_{k_1}(x,y)-(k^2_p-k^2_2)\gamma_{k_2}(x,y)\right]p(y)ds_y.
    \enn
\end{theorem}

\subsection{Proof of Theorem~\ref{jump}}
\label{sec:4.3}

Thanks to the derived regularized formulations of the integral operators, now we can prove the jump conditions stated in Theorem~\ref{jump}.

The jump condition for the single-layer potential follows trivially. Relating to the proof of Theorem~\ref{regK}, it can be concluded that $\mathcal{D}(\varphi)(z), x\in\Gamma$ takes a form similar to (\ref{KP}) for $z=x\pm h\nu_x\notin\Gamma$, i.e.,
\ben
\mathcal{D}= \begin{bmatrix}
	\mathcal{D}_1& \mathcal{D}_2\\
	\mathcal{D}_3 & \mathcal{D}_4
\end{bmatrix}, \quad \mathcal{D}_j=\mathcal{D}_j^1+\mathcal{D}_j^2M(\partial,\nu)+\left\{ {\mathcal{D}_j^3M(\partial,\nu)} \right\}^\top,\quad j=1,\cdots,4.
\enn
Letting $h\rightarrow 0^+$ (i.e. $z\rightarrow x\in\Gamma$), and applying the classical jump relations for acoustic and elastic problems~\cite{HW08}, it is only necessary to study the jumps of $\mathcal{D}_1^1$, $\mathcal{D}_2^3$, $\mathcal{D}_3^2$ and $\mathcal{D}_4^1$ in the forms associated with $K_1^1$, $K_2^3$, $K_3^2$ and $K_4^1$, respectively. Note that for $x\in\Gamma$,
\ben
\lim_{h\rightarrow 0^+,z=x\pm h\nu_x} \int_\Gamma \nabla_y\gamma_{k_\xi}(z,y)\varphi(y)ds_y= \pm \frac{\nu_x}{2}\varphi(x)+ \int_\Gamma \nabla_y\gamma_{k_\xi}(x,y)\varphi(y)ds_y,
\enn
which implies
\ben
&&\lim_{h\rightarrow 0^+,z=x\pm h\nu_x} \int_\Gamma \partial_{\nu_y}\gamma_{k_s}(z,y)u(y)ds_y= \pm\frac{1}{2}u(x)+\int_\Gamma \pa_{\nu_y}\gamma_{k_s}(x,y)u(y)ds_y,\\
&&-\lim_{h\rightarrow 0^+,z=x\pm h\nu_x} \int_\Gamma\partial_{\nu_y}\left[\frac{k_p^2-k_1^2}{k_1^2-k_2^2}\gamma_{k_1}(z,y) -\frac{k_p^2-k_2^2}{k_1^2-k_2^2}\gamma_{k_2}(z,y)\right]p(y)ds_y\\
&& =\pm\frac{1}{2}p(x)+\int_\Gamma \pa_{\nu_y}\left[\frac{k_p^2-k_1^2}{k_1^2-k_2^2}\gamma_{k_1}(x,y) -\frac{k_p^2-k_2^2}{k_1^2-k_2^2}\gamma_{k_2}(x,y)\right]p(y)ds_y.
\enn
It follows that
\ben
\lim_{h\rightarrow 0^+,z=x\pm h\nu_x} \mathcal{D}_1^1(u)(z) = \pm\frac{1}{2}u(x)+K_1^1(u)(x), \quad x\in\Gamma,
\enn
\ben
\lim_{h\rightarrow 0^+,z=x\pm h\nu_x} \mathcal{D}_2^3(p)(z) = K_2^3(p)(x), \quad x\in\Gamma,
\enn
\ben
\lim_{h\rightarrow 0^+,z=x\pm h\nu_x} \mathcal{D}_3^2(u)(z) = K_3^2(u)(x), \quad x\in\Gamma,
\enn
and
\ben
\lim_{h\rightarrow 0^+,z=x\pm h\nu_x} \mathcal{D}_4^1(p)(z) = \pm\frac{1}{2}p(x)+K_4^1(p)(x), \quad x\in\Gamma.
\enn
Therefore, we have
\ben
\lim_{h\rightarrow 0^+,z=x\pm h\nu_x} \mathcal{D}(\varphi)(z)= \pm\frac{1}{2}\varphi(x) +K(\varphi)(x),\quad x\in\Gamma.
\enn

Analogously, it can be derived that $\widetilde{T}(\pa_z,\nu_x)\mathcal{S}(\varphi)(z), x\in\Gamma$ takes a form similar to (\ref{KPP}) for $z=x\pm h\nu_x\notin\Gamma$. Due to the fact that for $x\in\Gamma$,
\ben
&&\lim_{h\rightarrow 0^+,z=x\pm h\nu_x} \int_\Gamma \nu_x\cdot\nabla_z\gamma_{k_s}(z,y)u(y)ds_y= \mp\frac{1}{2}u(x)+\int_\Gamma \pa_{\nu_x}\gamma_{k_s}(x,y)u(y)ds_y,\\
&&-\lim_{h\rightarrow 0^+,z=x\pm h\nu_x} \int_\Gamma \nu_x\cdot\nabla_z \left[\frac{k_p^2-k_1^2}{k_1^2-k_2^2}\gamma_{k_1}(z,y) -\frac{k_p^2-k_2^2}{k_1^2-k_2^2}\gamma_{k_2}(z,y)\right]p(y)ds_y\\
&&= \mp\frac{1}{2}p(x)+\int_\Gamma \pa_{\nu_x}\left[\frac{k_p^2-k_1^2}{k_1^2-k_2^2}\gamma_{k_1}(x,y) -\frac{k_p^2-k_2^2}{k_1^2-k_2^2}\gamma_{k_2}(x,y)\right]p(y)ds_y,
\enn
we arrive at the jump conditions
\ben
\lim_{h\rightarrow 0^+,z=x\pm h\nu_x} \widetilde{T}(\pa_z,\nu_x)\mathcal{S}(\varphi)(z)=\mp\frac{1}{2}\varphi(x) +K'(\varphi)(x),\quad x\in\Gamma.
\enn

It remains to prove the jump conditions for $\widetilde T(\partial_z,\nu_x)\mathcal{D}(\varphi)(z)$ as $h\rightarrow 0^+$. We can write it as
\ben
\widetilde T(\partial_z,\nu_x)\mathcal{D}(\varphi)(z)=\begin{bmatrix}
	\mathcal{N}_1& \mathcal{N}_2\\
	\mathcal{N}_3 & \mathcal{N}_4
\end{bmatrix}(\varphi)(z),
\enn
and it can be proved that $\mathcal{N}_j,j=1,\cdots,4$ take forms similar to $N_j,j=1,\cdots,4$, respectively. We first study $\mathcal{N}_1$ in the form
\ben
\mathcal{N}_1(u)(z)=&&\mathcal{N}^1_1(u)(z)+M(\partial_z,\nu_x)\mathcal{N}^2_1(M(\partial,\nu)u)(z)+ \tau_2^{z,x}\mathcal{N}^3_1(\tau_1u)(z)\\
&&+M(\partial_z,\nu_x)\mathcal{N}^4_1(u)(z)+\mathcal{N}^5_1(M(\partial,\nu)u)(z),
\enn
where $\tau_2^{z,x}u(z)=(\nu_x\times\nabla^S_z)\cdot u(z)$. Obviously,
\ben
\lim_{h\rightarrow 0^+,z=x\pm h\nu_x} \mathcal{N}^j_1(u)(z)= N^j_1(u)(x),\quad x\in\Gamma, j=1,2,3,
\enn
and thus,
\ben
\lim_{h\rightarrow 0^+,z=x\pm h\nu_x} M(\partial_z,\nu_x)\mathcal{N}^2_1(M(\partial,\nu)u)(z)= M(\partial_x,\nu_x)N^2_1(M(\partial,\nu)u)(x),\quad x\in\Gamma,
\enn
\ben
\lim_{h\rightarrow 0^+,z=x\pm h\nu_x} \tau_2^{z,x}\mathcal{N}^3_1(\tau_1u)(z)= \tau_2N^3_1(\tau_1u)(x),\quad x\in\Gamma.
\enn
In addition, note that
\ben
\mathcal{N}^4_1(u)(z)&= & \int_\Gamma \nabla _y \left[-2\mu(\gamma_{k_s}(z,y)-\gamma_{k_1}(z,y) )+C_3(\gamma_{k_1}(z,y)-\gamma_{k_2}(z,y) )\right] \nu^\top_y u(y)ds_y\\&\quad&
	+\mu\int_\Gamma\partial_{\nu_y}\gamma_{k_s}(z,y) u(y)ds_y,\\
\mathcal{N}^5_1(u)(z)&= &\int_\Gamma \nu_x\nabla _z^\top \left[-2\mu(\gamma_{k_s}(z,y)-\gamma_{k_1}(z,y) )+C_4(\gamma_{k_1}(z,y)-\gamma_{k_2}(z,y) )\right]u(y)ds_y\\&\quad&
	+\mu\int_\Gamma\nu_x^\top\nabla_z\gamma_{k_s}(z,y) u(y)ds_y,
\enn
we have
\ben
\lim_{h\rightarrow 0^+,z=x\pm h\nu_x} \mathcal{N}^4_1(u)(z)= \pm \frac{\mu}{2}u(x) +N^4_1(u)(x),\quad x\in\Gamma,
\enn
and
\ben
\lim_{h\rightarrow 0^+,z=x\pm h\nu_x} \mathcal{N}^5_1(u)(z)= \mp \frac{\mu}{2}u(x) +N^5_1(u)(x),\quad x\in\Gamma.
\enn
Therefore,
\ben
&&\lim_{h\rightarrow 0^+,z=x\pm h\nu_x} M(\partial_z,\nu_x)\mathcal{N}^4_1(u)(z)+\mathcal{N}^5_1(M(\partial,\nu)u)(z)\\
&=& \pm \frac{\mu}{2}M(\partial_x,\nu_x)u(x) +M(\partial_x,\nu_x)N^4_1(u)(x)\\
&\quad& \mp \frac{\mu}{2}M(\partial_x,\nu_x)u(x)+ N^5_1(M(\partial,\nu)u)(x)\\
&=& M(\partial_x,\nu_x)N^4_1(u)(x)+ N^5_1(M(\partial,\nu)u)(x),\quad x\in\Gamma,
\enn
which gives that
\ben
\lim_{h\rightarrow 0^+,z=x\pm h\nu_x} \mathcal{N}_1(u)(z)=N_1(u)(x),\quad x\in\Gamma.
\enn
The proof of the other three jump conditions
\ben
\lim_{h\rightarrow 0^+,z=x\pm h\nu_x} \mathcal{N}_j(u)(z)=N_j(u)(x),\quad \quad x\in\Gamma, j=2,3,4,
\enn
is analogous and hence is omitted. This completes the proof of Theorem~\ref{jump}.$\hfill\square$

\section{Regularized boundary integral equation solver}
\label{sec:5}

\subsection{Regularized boundary integral equations}
\label{sec:5.1}

Making use of the spectral properties of the poroelastic integral operators presented in Section~\ref{sec:3.3}, we are able to construct regularized boundary integral equations (RBIEs) with favorable features of better spectral properties. According to the regularized integral equation method discussed in~\cite{ZXY21} for two-dimensional poroelastic problems, we can choose the static single-layer operator $\mathcal{R}=S_0$ given by
\be
S_0(\varphi)(x):=\int_\Gamma(E_0(x,y))^\top\varphi(y)ds_y, \quad x\in\Gamma,
\en
as the regularized operator.

\begin{figure}[htb]
	\centering
	\begin{tabular}{cc}
		\includegraphics[scale=0.3]{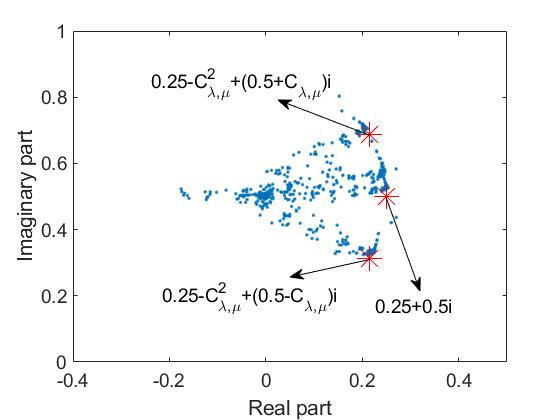} &
		\includegraphics[scale=0.3]{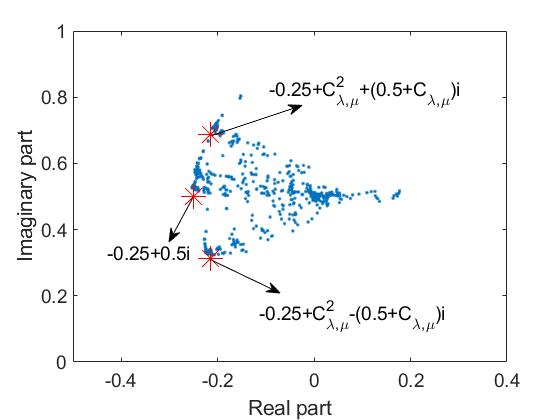} \\
		(a) $i\eta\left(\frac{1}{2}I-K\right)-\mathcal{R}N$ & (b) $i\eta( \frac{I}{2}-K')+N\mathcal{R}$  \\
	\end{tabular}
	\caption{Eigenvalue distributions of the operators $i\eta\left(\frac{1}{2}I-K\right)-\mathcal{R}N$ (a) and $i\eta( \frac{I}{2}-K')+N\mathcal{R}$ (b) for a unit ball scatterer.}
	\label{Reeig}
\end{figure}

Therefore, for the direct method, the DCBIE (\ref{DCBIE}) can be regularized as (called DRBIE)
\be
\label{DRBIE}
\left[i\eta\left(\frac{1}{2}I-K\right)-\mathcal{R}N\right]U(x)=-\left[\mathcal{R}(\frac{1}{2}I+K')+i\eta S\right](F(x)),\quad x\in \Gamma.
\en
For the indirect method, replacing the solution representation (\ref{CSDLP}) by
\be
\label{RR}
U(x)=(\mathcal{DR}-i\eta\mathcal{S})(\varphi)(x), \quad x\in \Omega^c, \quad \eta\neq 0,
\en
we can obatin the regularized form of ICBIE (called IRBIE) as follows
\be
\label{IRBIE}
\left[i\eta( \frac{I}{2}-K')+N\mathcal{R}\right](\varphi)=F \quad \rm on\quad \Gamma,
\en
instead of the classical ICBIE (\ref{ICBIE}). In view of the spectra results in Section~\ref{sec:3.3}, we can conclude that the spectrum of the regularized integral operator on the left-hand side of DRBIE (\ref{DRBIE}) consists of three nonempty sequences of eigenvalues which converge to $1/4-i\eta/2$, $1/4-C^2_{\lambda,\mu}+i\eta(1/2+C_{\lambda,\mu})$ and $1/4-C^2_{\lambda,\mu}+i\eta(1/2-C_{\lambda,\mu})$, respectively. Similarly, we can also observe that the eigenvalues of the integral operator on the left-hand side of IRBIE (\ref{IRBIE}) accumulated at $-1/4+i\eta/2$, $-1/4+C^2_{\lambda,\mu}+i\eta(1/2+C_{\lambda,\mu})$ and $-1/4+C^2_{\lambda,\mu}+i\eta(1/2-C_{\lambda,\mu})$, see for example Figure~\ref{Reeig}, a numerical verification.

\subsection{Numerical discretization}
\label{sec:5.2}

In this section, we briefly introduce the application of the Chebyshev-based rectangular-polar solver discussed in \cite{BG20,BY20} to the numerical discretization the poroelastic BIOs. Based on a partition of the boundary using non-overlapping parametric curvilinear patches, this approach interpolates the unknowns on a Chebyshev grid on each patch in terms of Chebyshev polynomials. For the corresponding acceleration of this method, we refer to~\cite{JBB}.

Let $\Gamma$ be partitioned into a set of $M$ non-overlapping parametrized (logically-rectangular) patches $\Gamma_q, q=1,...,M$ as
\ben
\Gamma=\bigcup_{q=1}^M \Gamma_q, \quad \Gamma_q=\left\{ \textbf{r}^q(u,v)=(x^q(u,v),y^q(u,v),z^q(u,v))^\top:\left[-1,1\right]^2\to \R^3 \right\}.
\enn
Introducing the tangential covariant basis vectors and surface normal on $\Gamma_q:$
\ben
a_u^q=\frac{\partial \textbf{r}^q(u,v)}{\partial u},\quad a_v^q=\frac{\partial \textbf{r}^q(u,v)}{\partial v}, \quad \nu^q=\frac{a_u^q\times a_v^q}{\left| a_u^q\times a_v^q \right|},
\enn
we can obtain the metric tensor as
\ben
{G^q} = \begin{bmatrix}
	g^q_{uu}&g^q_{uv}\\
	g^q_{vu}&g^q_{vv}
	\end{bmatrix},
\enn
where $g^q_{ij}=a^q_i\cdot a^q_j$ and thus, the surface element Jacobian is given by
\ben
ds=J^q(u,v)dudv=\sqrt {\left|G^q  \right|}dudv.
\enn
Here, $\left|G^q  \right|$ is the determinant of $G^q$. As a result, the surface gradient of a given density $\varphi=\varphi(\textbf{r}^q(u,v))$ can be expressed as
\ben
\nabla_x^S\varphi=\sum\limits_{i,j=1}^{2} g^{ij}\partial_i\varphi\partial_j \textbf{r}^q(u_i,v_j),\quad \partial_1=\frac{d}{du},\partial_2=\frac{d}{dv},
\enn
where $g^{ij},i,j=1,2$ denote the components of the inverse of the matrix $G^q$.

Given a density $\varphi$, it can be approximated on $\Gamma_q$ by the Chebyshev polynomials as
\ben
\varphi (x) = \sum\limits_{i,j = 0}^{N - 1} \varphi _{ij}^q a_{ij}(u,v), \quad x\in \Gamma_q,
\enn
where
\ben
a_{ij}(u,v)=\frac{1}{N^2}\sum^{N-1}_{m,n=0} \alpha_n\alpha_mT_n(u_i)T_m(v_j)T_n(u)T_m(v),\quad {\alpha _n} = \begin{cases}
	1, & n=0,\cr
	2, & n\neq 0,
	\end{cases}
\enn
and the coefficients $\varphi _{ij}^q=\varphi(x^q_{ij})$ denote the values of the continuous density $\varphi$ at the discretization points $x^q_{ij}=\textbf{r}^q(u_i,v_j)$ with
\ben
u_i=\cos\left(\frac{2i+1}{2N}\pi\right),\quad v_j=\cos\left(\frac{2j+1}{2N}\pi\right),\quad i,j=0,...,N-1.
\enn
Therefore, we can obtain the approximation of the surface gradient as
\ben
(\nabla_x^S\varphi)\Big|_{x=x^q_{ij}}=\sum\limits_{n,m=0}^{N-1}B_{ij,nm}^q\varphi_{nm}^q, \quad B_{ij,nm}^q=\left(\sum\limits_{i,j=1}^{2} g^{ij}\partial_ia_{nm}\partial_j \textbf{r}^q  \right)\Big|_{u=u_i,v=v_j}.
\enn

We now discuss the discretizations of the prorelastic BIOs. On a basis of the regularized formulations of integral operators given in Sections~\ref{sec:4}, the numerical implementations can be converted into evaluating multiple operators of two types, (i)Integral operators with weakly-singular kernels $H(x,y)$
\ben
\mathcal{H}\varphi(x)=\int_\Gamma H(x,y)\varphi(y)ds_y,
\enn
and (ii) surface-differentiation operators $M(\pa,\nu)$, $\tau_1$, $\tau_2$ appearing in Section~\ref{sec:4}, which can be extracted from the approximation of surface gradient $\nabla_x^S$. Clearly, the integrals $\mathcal{H}\varphi(x)$ over $\Gamma$ can be split into the sum of integrals over each of the $M$ patches,
\ben
\mathcal{H}\varphi(x)=\sum\limits_{q=1}^{M} \mathcal{H}_q(x),\quad \mathcal{H}_q(x):=\int_{\Gamma_q}H(x,y)\varphi(y)ds_y,\quad x\in \Gamma.
\enn

In the ``non-adjacent'' integration case, in which the target point $x^{\tilde q}_{ij}$ is far from the integration patch, the integral $\mathcal{H}_q(x^{\tilde q}_{ij})$ is non-singular. Then the classical Fej\'er’s first quadrature rule can be utilized to obtain the approximation
\ben
\label{non-ad}
\mathcal{H}_q(x^{\tilde q}_{ij})\approx \sum\limits_{m,n=0}^{N-1} A^{\tilde q,q}_{ij,nm} \varphi_{nm}^q,
\enn
where
\ben
A^{\tilde q,q}_{ij,nm}=H(x^{\tilde q}_{ij},\textbf r^q(u_n,v_m))J^q(u_n,v_m)w_nw_m,
\enn
with the quadrature weights
\ben
w_j=\frac{2}{N}\left(1-2\sum\limits_{l=1}^{\left\lfloor N/2 \right\rfloor }\frac{1}{4l^2-1}\cos(lu_j)\right), \quad j=0,...,N-1.
\enn

In the ``adjacent'' integration case, in which the point $x^{\tilde q}_{ij}$ either lies within the integration patch or is located very close to it, the integral $\mathcal{H}_q(x^{\tilde q}_{ij})$ becomes weakly-singular and nearly-singular, respectively. It is suggested in \cite{BG20} to constructed a new graded mesh, relying on a smoothing change of variables $\xi_{u}(s)$ (for more details, see \cite{BG20,BY20}), around the point which lies closest to $x^{\tilde q}_{ij}$ with parameters
\ben
(\tilde{u}^q,\tilde{v}^q)=\mbox{arg}\mbox{min}_{(u,v)\in\left[-1,1\right]^2}\left\{ \left|x^{\tilde{q}}_{ij}-\textbf{r}^q(u,v) \right| \right\}.
\enn
Then the single integral can be approximated by
\ben
\label{ad}
\nonumber
\mathcal{H}_q(x^{\tilde q}_{ij})
& \approx &\sum\limits_{n,m=0}^{N-1} \varphi_{nm}^q\int_{-1}^{1}\int_{-1}^{1}H(x^{\tilde q}_{ij},\textbf{r}(u,v))J^q(u,v)a_{nm}(u,v)dudv \\
\nonumber
&=&\sum\limits_{n,m=0}^{N-1} \varphi_{nm}^q\int_{-1}^{1}\int_{-1}^{1}\tilde{H}(x^{\tilde q}_{ij},s,t)\tilde{J}^q(s,t)\tilde{a}_{nm}(s,t)\xi_{\tilde{u}^q}^{'}(s)\xi_{\tilde{v}^q}^{'}(t)dsdt\\
&\approx&\sum\limits_{n.m=0}^{N-1}C_{ij,nm}^{\tilde{q},q}\varphi_{nm}^q,
\enn
where
\ben
&&C_{ij,nm}^{\tilde{q},q}=\sum\limits_{n.m=0}^{N-1}H(x^{\tilde q}_{ij},\textbf{r}^q(\xi_{\tilde{u}^q}(s),\xi_{\tilde{v}^q}(t))) J^q(\xi_{\tilde{u}^q}(s),\xi_{\tilde{v}^q}(t)) \\
&& \quad\quad\quad\quad\quad\quad\quad\times a_{nm}(\xi_{\tilde{u}^q}(s),\xi_{\tilde{v}^q}(t)) \xi_{\tilde{u}^q}^{'}(\tilde{t}_{l_1})\xi_{\tilde{v}^q}^{'}(\tilde{t}_{l_2})\tilde{w}_{l_1}\tilde{w}_{l_2}.
\enn
Here, the quadrature nodes and weights for an order $N^\beta$ are given analogous to the ``non-adjacent'' case.

\section{Numerical experiments}
\label{sec:6}

In this section, several numerical examples, involving three bounded obstacles depicted in Fig.~\ref{scatter}, are presented to demonstrate the accuracy and efficiency of the proposed methods for solving three-dimensional poroelastic problems. Utilizing the dimensionless technique discussed in~\cite{CD95,S05}, we set $\mu=2$, $\nu_p=0.2$, $\nu_u=0.33$, $B=0.62$, $C=0.66$, $\phi=0.333$, $\rho_s=1$, $\rho_f=0.5$, $\kappa=1$. We additionally choose $\eta=1$. Here, we use the fully complex version of the iterative solver GMRES to produce the solutions of the integral equations and the maximum errors defined by
\be
\epsilon_{\infty}&:=\frac{\mbox{max}_{x\in S}{\left|U^{num}(x)-U^{exa}(x)\right|}}{\mbox{max}_{x\in S}{\left|U^{exa}(x)\right|}},
\en
will be displayed. Here, $S$ is the square $\left[-2,2\right]\times \left[-2,2\right]\times \left\{ {2} \right\}  \subset \Omega^c$, $U^{exa}$ is the exact solution of the poroelastic problem (\ref{model}), and $U^{num}$ is the numerical solution generated from the DCBIE (\ref{DCBIE}), ICBIE (\ref{ICBIE}), DRBIE (\ref{DRBIE}) or IRBIE (\ref{IRBIE}), respectively. The particular implementation for the numerical experiments is programmed in Fortran and is parallelized using OpenMP.

\begin{figure}[htb]
	\centering
	\begin{tabular}{ccc}
		\includegraphics[scale=0.18]{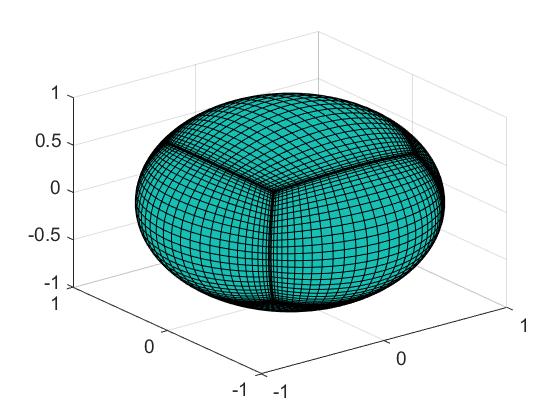} &
		\includegraphics[scale=0.18]{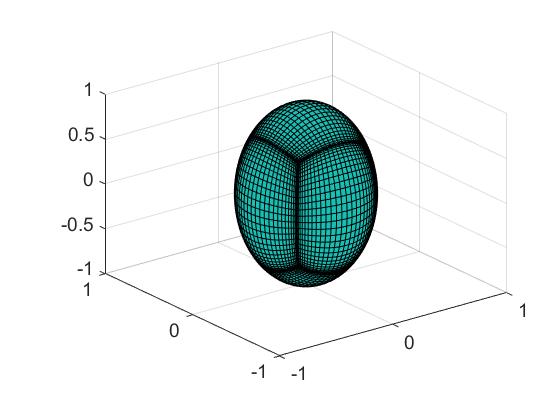} &
		\includegraphics[scale=0.18]{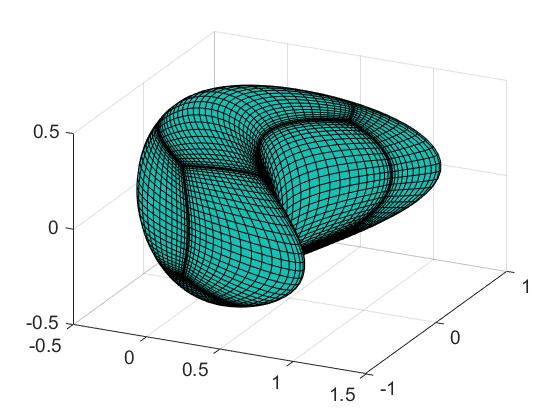} \\
		(a) Ball & (b) Ellipsoid & (c) Bean \\
	\end{tabular}
	\caption{Obstacles considered in the numerical tests.}
	\label{scatter}
\end{figure}

We first test the accuracy of the proposed methods. The exact solution $U^{exa}=({u^{exa}}^\top,p^{exa})^\top$ is given by
\ben
u^{exa}(x)=E_{21}(x,z), \qquad p^{exa}(x)=E_{22}(x,z), \qquad x\in \Omega^c,
\enn
with $z=(0,0.5,0.3)$ for the obstacle Fig.~\ref{scatter}(a) and $z=(0,0,0)$ for the obstacles Fig.~\ref{scatter}(b,c), which gives the boundary data $F=\widetilde T(\partial,\nu)U^{exa}$ on $\Gamma$. We first consider $\omega=\pi$ and the Chebyshev grid with $M=6$, $N=16$ and $N_\beta=100$. Fig.~\ref{direct exa} shows the errors $|U^{num}(x)-U^{exa}(x)|$ between the numerical solution $U^{num}$ resulting from solving DRBIE (\ref{DRBIE}) and the exact solution $U^{exa}$ for the obstacle Fig.~\ref{scatter}(a) and $x\in\{x\in\R^3: |x|=2\}$ with a maximum value $1.814\times 10^{-8}$. Next, we consider the poroelastic problem of scattering by obstacle Fig.~\ref{scatter}(b) on a basis of six $2\times 2$ patches ($M=24$) with $\omega=20$, $N=16$, $N_\beta=200$ and employ IRBIE (\ref{IRBIE}), the point-wise values of the numerical and exact solutions on the line segment $\{x\in\R^3: x_1=2, x_2\in[-2,2], x_3=2\}$ are displayed in Fig.~\ref{indirect exa}. The relative error for this case is $\epsilon_{\infty}=6.67\times 10^{-6}$. In Fig.~\ref{NE}, the numerical errors $\epsilon_{\infty}$ with respect to $N$ using DRBIE and IRBIE for the obstacles Fig.~\ref{scatter} are presented while choosing $\omega=2\pi$ and Chebyshev grid with $M=6$. Higher accuracy can be achieved by increasing the parameter $N^\beta$ and treating the evaluation of weakly-singular kernels for small $|x-y|$ with cares.

\begin{figure}[htb]
	\centering
	\begin{tabular}{c}
		\includegraphics[scale=0.35]{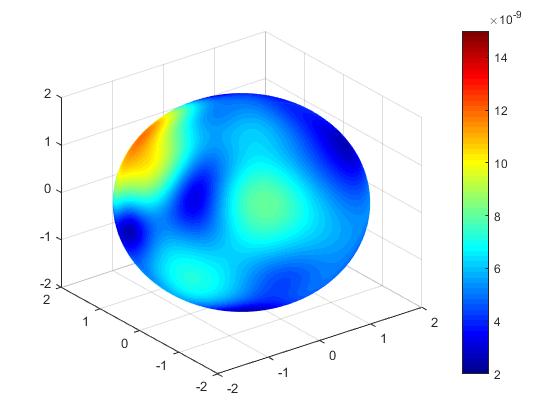}
	\end{tabular}
	\caption{The numerical errors on a sphere for the poroelastic problem of sacttering by the obstacle Fig.~\ref{scatter}(a).}
	\label{direct exa}
\end{figure}

\begin{figure}[htb]
	\centering
	\begin{tabular}{cccc}
		\includegraphics[scale=0.14]{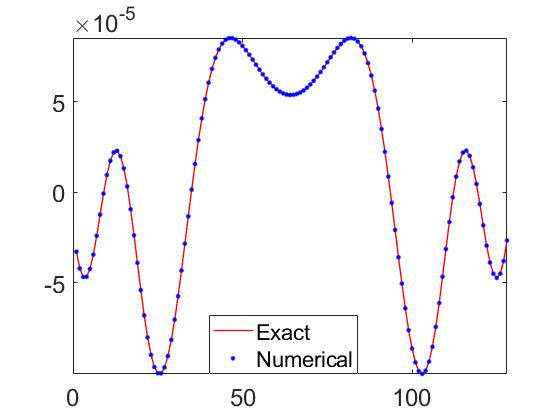} &
		\includegraphics[scale=0.14]{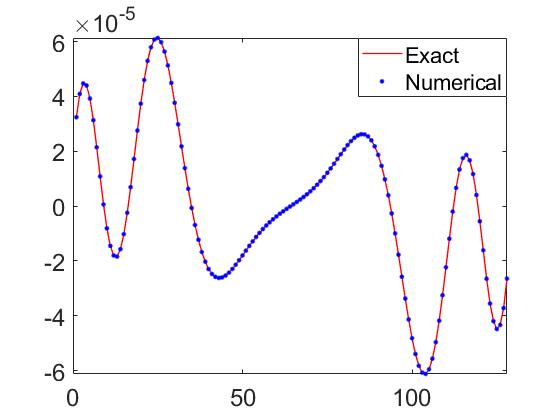} &
		\includegraphics[scale=0.14]{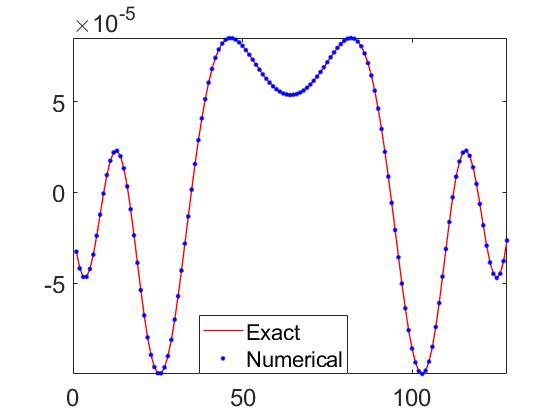} &
		\includegraphics[scale=0.14]{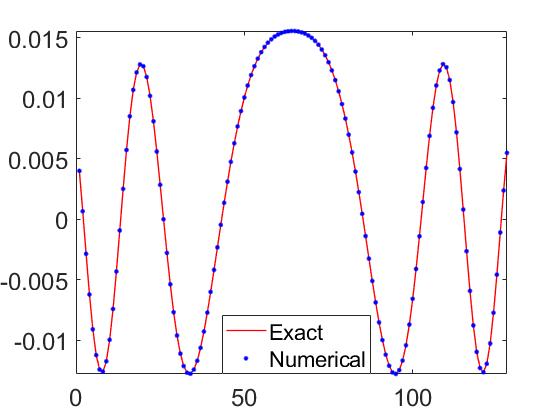} \\
		(a) $\mbox{Re}(u_1)$ & (b) $\mbox{Re}(u_2)$ &(c)$\mbox{Re}(u_3)$&(d)$\mbox{Re}(p)$  \\
		\includegraphics[scale=0.14]{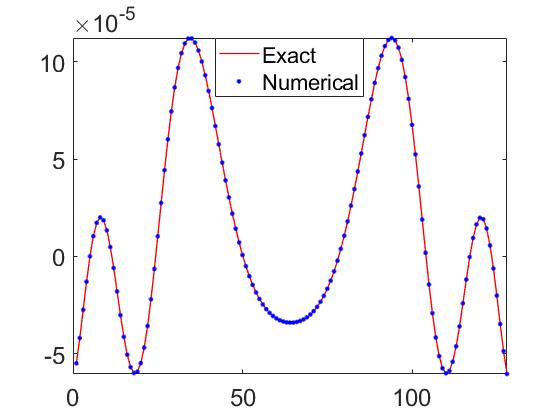} &
		\includegraphics[scale=0.14]{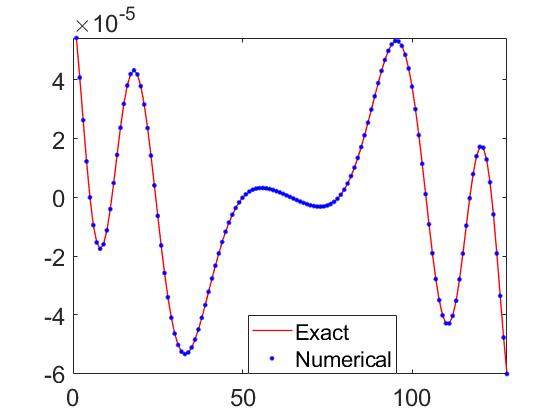} &
		\includegraphics[scale=0.14]{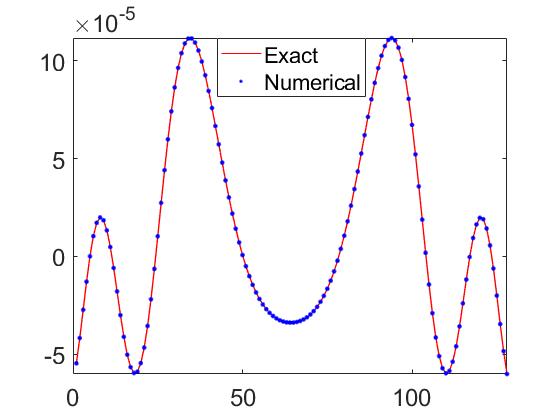} &
		\includegraphics[scale=0.14]{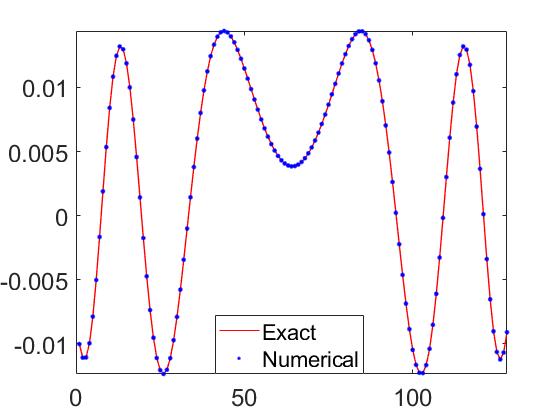} \\
		(e) $\mbox{Im}(u_1)$ & (f) $\mbox{Im}(u_2)$ &(g)$\mbox{Im}(u_3)$&(h)$\mbox{Im}(p)$  \\
	\end{tabular}
	\caption{Comparison of the exact and numerical solutions for the poroelastic problem of scattering by obstacle Fig.~\ref{scatter}(b).}
	\label{indirect exa}
\end{figure}

\begin{figure}[htb]
	\centering
	\begin{tabular}{ccc}
		\includegraphics[scale=0.18]{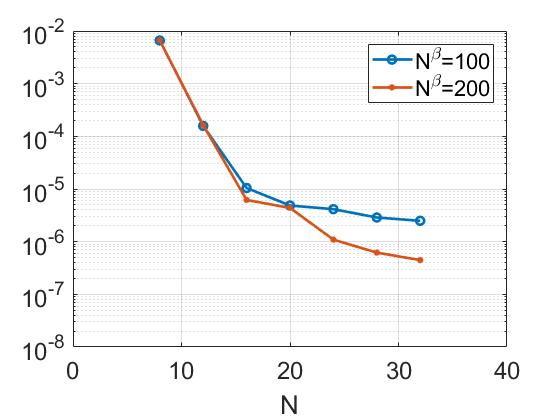} &
		\includegraphics[scale=0.18]{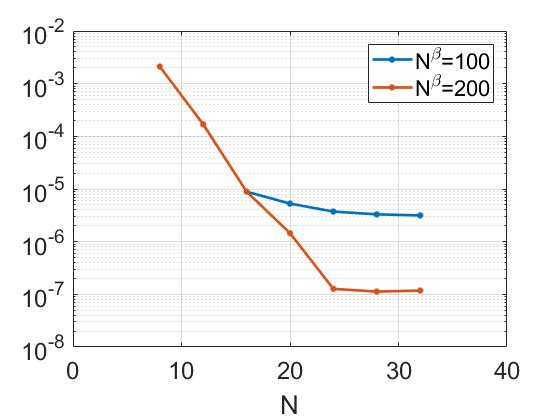} &
		\includegraphics[scale=0.18]{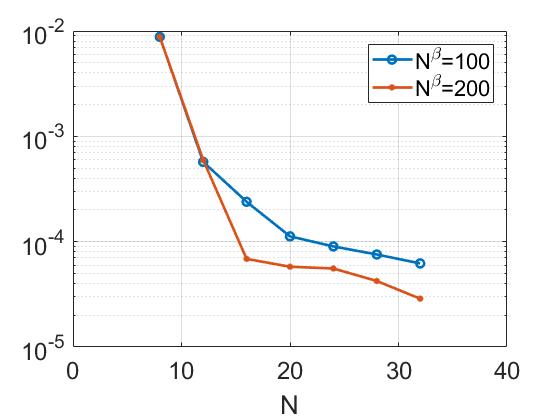} \\
		(a) DRBIE (\mbox{Ball}) & (b) DRBIE (\mbox{Ellipsoid}) &(c)DRBIE (\mbox{Bean})  \\
		\includegraphics[scale=0.18]{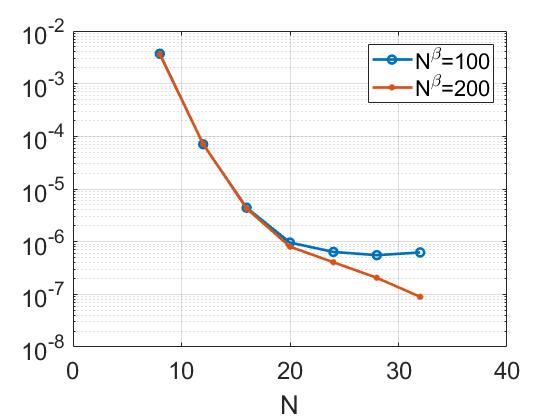} &
		\includegraphics[scale=0.18]{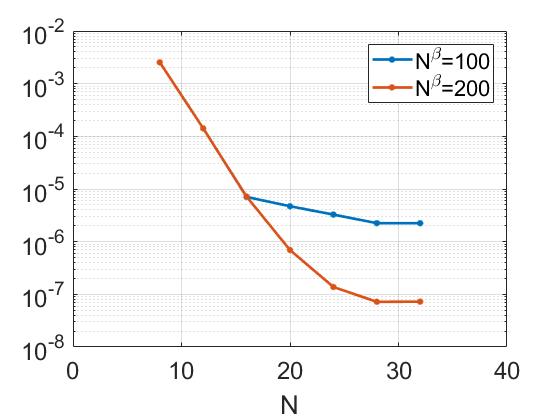} &
		\includegraphics[scale=0.18]{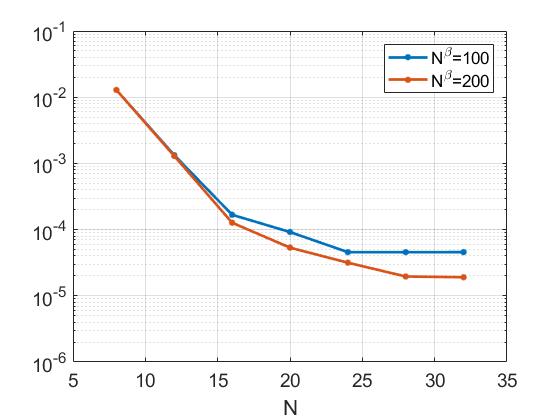} \\
		(e) IRBIE (\mbox{Ball}) & (f) IRBIE (\mbox{Ellipsoid}) &(g)IRBIE (\mbox{Bean})  \\
	\end{tabular}
	\caption{Numerical errors $\epsilon_{\infty}$ for the problem of scattering by the the obstacles Fig.~\ref{scatter}. }
	\label{NE}
\end{figure}

Next, we verify the efficiency of the regularized integral equation methods. Choosing $\omega=2\pi$ and the Chebyshev grid with $M=6$, $N=32$ and $N_\beta=200$, Fig.~\ref{GMERES} displays the history of GMRES residuals as functions of the number of iterations for the method of using DCBIE (\ref{DCBIE}), ICBIE (\ref{ICBIE}), DRBIE (\ref{DRBIE}) and IRBIE (\ref{IRBIE}), respectively. The rapid convergence results of regularized methods demonstrate that use of the regularized integral equations is highly beneficial compared to the un-regularized ones. With $\omega=20$ and $M=24$,  Table~\ref{timedi} lists the precomputation time, time per iteration and number of iterations required by the regularized integral equation methods. For the DRBIE method, an accuracy of $7.6\times 10^{-3}$ $(\mbox{resp}. 2.5\times 10^{-6})$ can be achieved by setting $N=8$ $(\mbox{resp}. N=16)$, while an accuracy of $2.1\times 10^{-2}$ $(\mbox{resp}. 1.7\times 10^{-6})$ can be obtained for the IRBIE method.

\begin{figure}[htb]
	\centering
	\begin{tabular}{cc}
		\includegraphics[scale=0.3]{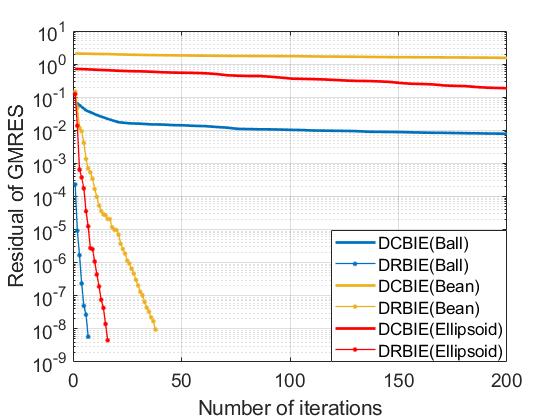} &
		\includegraphics[scale=0.3]{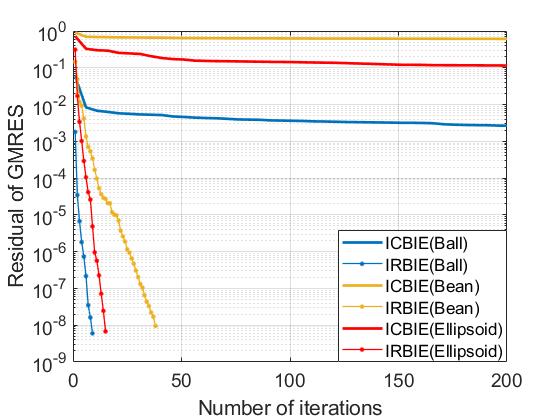} \\
		(a) DRBIE  & (b) IRBIE
	\end{tabular}
	\caption{GMRES residual $\epsilon_r$ for the problem of scattering by the obstacles Fig.~\ref{scatter}. }
	\label{GMERES}
\end{figure}

\begin{table}[htb]
	\label{timedi}
	\centering
	\caption{Computation times and number of iterations required by the DRBIE method.}
	\begin{tabular}{|c|c|c|c|c|c|c|}
		\hline
		\multirow{2}{*}{$N$} &\multirow{2}{*}{$N_\beta$}& \multirow{2}{*}{$N_{DOF}$}  & \multicolumn{4}{|c|}{DRBIE} \\
		\cline{4-7}
		& && Time(prec.) & Time(1iter.) & Niter($\epsilon_r$) & $\epsilon_{\infty}$\\
		\hline
		8&100&$4\times 1536$  &4.7 s&11.7 s&67 ($9.9\times 10^{-6}$)&$9.1\times 10^{-3}$     \\
		8&200&$4\times 1536$  &18.3 s&11.6  s&49 ($9.8\times 10^{-6}$)&$7.6\times 10^{-3}$ \\
		16&100&$4\times 6144$&24.4 s&2.18 min&54 ($9.4\times 10^{-9}$)&$5.9\times 10^{-6}$  \\
		16&200&$4\times 6144$&1.46 min   &2.18 min&33 ($9.8\times 10^{-9}$) &$2.5\times 10^{-6}$  \\
		\hline
		\multirow{2}{*}{$N$} &\multirow{2}{*}{$N_\beta$}  & \multirow{2}{*}{$N_{DOF}$}&  \multicolumn{4}{|c|}{IRBIE} \\
		\cline{4-7}
		& && Time(prec.) & Time(1iter.) & Niter($\epsilon_r$) & $\epsilon_{\infty}$\\
		\hline
		8&100 &$4\times 1536$ &3.6 s   &25.5 s    &31 ($8.5\times 10^{-6}$ )  &$2.1\times 10^{-2}$   \\
		8&200 &$4\times 1536$ &13.8 s   &24.7 s   &30 ($8.6\times 10^{-6}$ )  & $2.1\times 10^{-2}$  \\
		16&100&$4\times 6144$&18.1 s   &2.14 min   &22 ($9.6\times 10^{-9}$ )  &$2.4\times 10^{-6}$     \\
		16&200&$4\times 6144$ &1.09 min &2.65 min  &20 ($8.5\times 10^{-9}$ ) &$1.7\times 10^{-6}$ \\
		\hline
	\end{tabular}
\end{table}

Finally, we consider the scattering of an incident point source $U^{inc}$ in the form
\ben
	U^{inc}=({u^{inc}}^\top,p^{inc})^\top,\quad u^{inc}(x)=E_{12}(x,z),\quad p^{inc}(x)=E_{22}(x,z)
\enn
by the obstacle Fig.~\ref{scatter}(b) where $z=(3,2,0)$ denotes the location of the point source. The numerical solutions in $\Omega^c$ with $\omega=20$ are presented in Figs.~\ref{dscattering} and \ref{scattering} based on the DRBIE and IRBIE, respectively. A total of 41 (resp. 32) iterations sufficed for the DRBIE (resp. IRBIE) method to reach the GMRES residual tolerance value $\epsilon_r=1\times 10^{-4}$. The numerical results demonstrate the accuracy and efficiency of the proposed regularized boundary integral equation methods.

\begin{figure}[htb]
	\centering
	\begin{tabular}{cc}
		\includegraphics[scale=0.25]{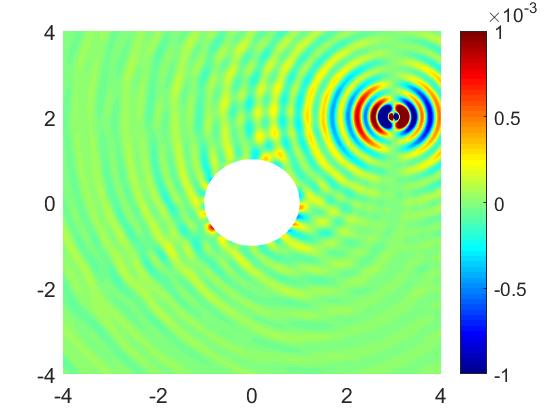} &
		\includegraphics[scale=0.25]{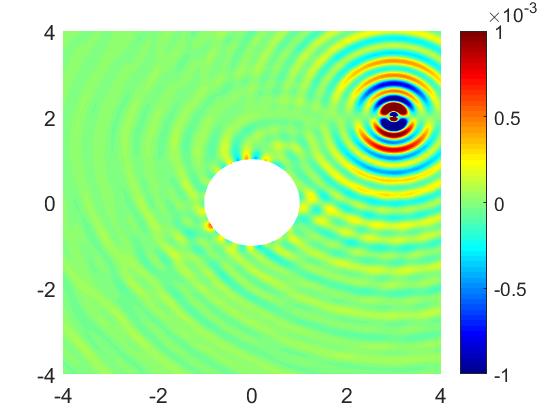} \\
		(a) Re($u_1^{num}$) & (b) Re($u_2^{num}$)  \\
		\includegraphics[scale=0.25]{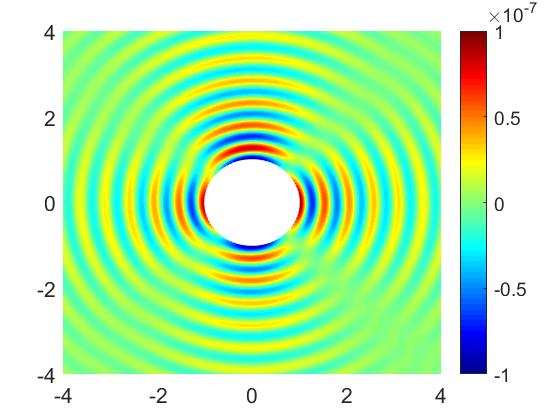} &
		\includegraphics[scale=0.25]{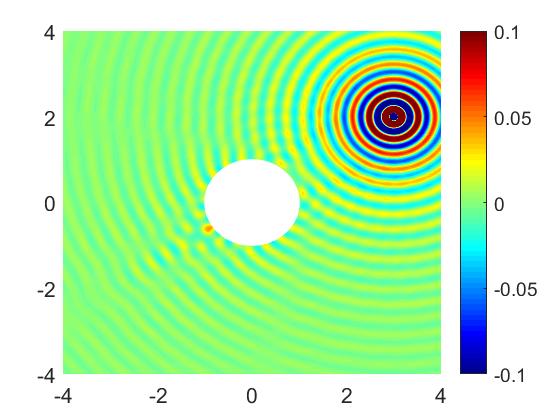} \\
		(e) Re($u_3^{num}$) & (f) Re($p^{num}$) \\
	\end{tabular}
	\caption{Real parts of the total field $U$ on an $x_3=0$ section for the scattering of an incident point source by the obstacle Fig.~\ref{scatter}(b).}
	\label{dscattering}
\end{figure}

\begin{figure}[htb]
	\centering
	\begin{tabular}{cc}
		\includegraphics[scale=0.25]{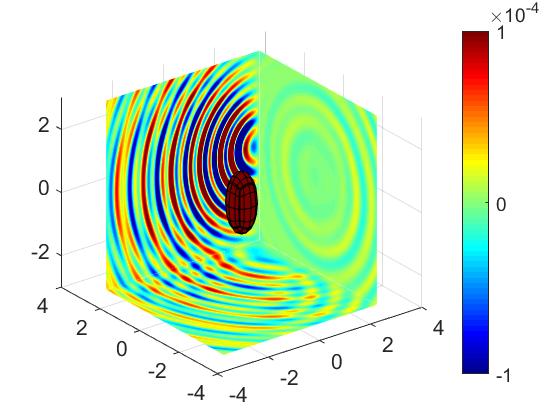} &
		\includegraphics[scale=0.25]{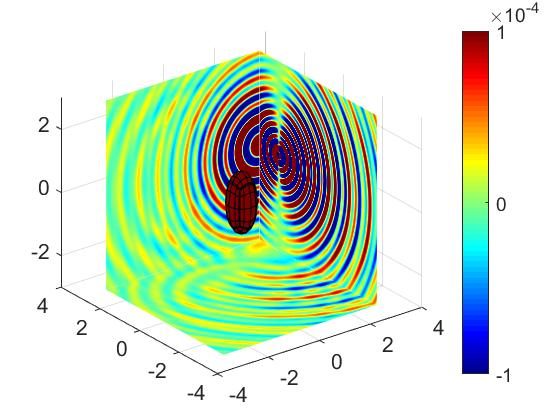} \\
		(a) Re($u_1^{num}$) & (b) Re($u_2^{num}$)  \\
		\includegraphics[scale=0.25]{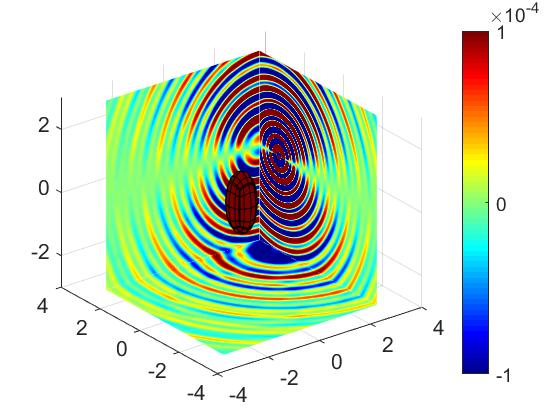} &
		\includegraphics[scale=0.25]{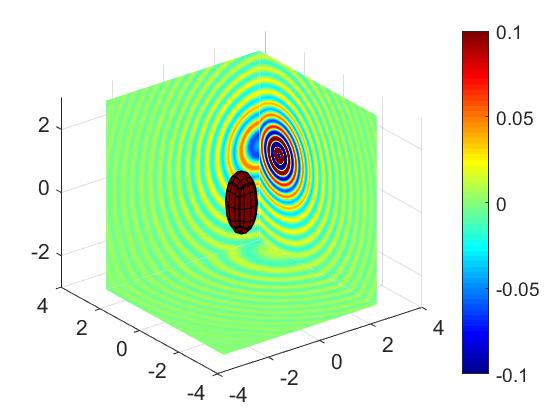} \\
		(e) Re($u_3^{num}$) & (f) Re($p^{num}$) \\
	\end{tabular}
	\caption{Real parts of the total field $U$ around the obstacle for the scattering of an incident point source by the obstacle Fig.~\ref{scatter}(b).}
	\label{scattering}
\end{figure}

\section*{Acknowledgement}
The work of LWX is supported by an NSFC Grant (No.12071060). The work of TY is supported by an NSFC Grant (No. 12171465).

\appendix
\section*{Appendix. Regularized expressions of the strongly-singular and hyper-singular operators and proofs.}
\renewcommand{\theequation}{A.\arabic{equation}}

This appendix presents the main approach for deriving the regularized formulations of the strongly-singular and hyper-singular operators. Analogous technique has been shown in \cite{BXY191} for the three-dimensional elastic and thermoelastic problems. But the derivations for three-dimensional poroelastic problem are more complex and thus, we present the full proof to make this appendix individually readable.

Given the G\"{u}nter derivative operator
\ben
M(\partial,\nu)u(x)=\partial_\nu u-\nu(\nabla  \cdot  u)+\nu  \times {\rm{curl}}\, u,
\enn
we can rewrite the traction operator $T(\partial ,\nu )$ as
\ben
T(\partial ,\nu )u(x) = (\lambda  + \mu )\nu (\nabla  \cdot u) + \mu {\partial _\nu }u + \mu M(\partial ,\nu )u.
\enn
Then
\ben
\label{td}
&&T(\partial ,\nu )\nabla=(\lambda+\mu)\nu\Delta+\mu\partial_\nu\nabla+\mu M(\partial,\nu)\nabla\\
&&=(\lambda+\mu)\nu\Delta+\mu\partial_\nu\nabla-\mu M(\partial,\nu)\nabla+2\mu M(\partial,\nu)\nabla,
\enn
together with
\be
{\partial _\nu }\nabla  - M(\partial ,\nu )\nabla  = \nu \Delta,
\label{eq.G1}
\en
imply that
\be
T(\partial ,\nu )\nabla  = (\lambda  + 2\mu )\nu \Delta  + 2\mu M(\partial ,\nu )\nabla.
\label{eq.G2}
\en
Letting $M(\partial_x,\nu_x)=\left[m_x^{ij}\right]^3_{i,j=1}$, it can be known that
\ben
m_x^{ij}=\partial_{x_i}\nu_x^j-\partial_{x_j}\nu_x^i=-m_x^{ji} \quad \mbox {for} \quad i,j=1,2,3.
\enn
Therefore, we have the following properties of the operator $M(\partial,\nu)$~\cite{HW08}. For any scalar fields $p$, $q$, vector fields $u$, $v$ and tensor field $\Pi$, there hold the Stokes formulas
\be
\label{s1}
\int_\Gamma(m^{ij}p)qds&=&-\int_\Gamma p(m^{ij}q) ds, \\
\label{s2}
\int_\Gamma(Mu)\cdot vds&=&\int_\Gamma u\cdot(Mv) ds, \\
\label{s3}
\int_\Gamma(Mq) vds&=&-\int_\Gamma q(Mv) ds, \\
\label{s4}
\int_\Gamma(M\Pi)^\top vds&=&\int_\Gamma \Pi^\top(Mv) ds.
\en
We first have the following result~\cite[Lemma 4.2]{BXY191}. For $x\ne y$, it follows that
\be
\label{Tx}
&&T(\partial_x,\nu_x)E_{11}(x,y)\nonumber\\
&&=-\nu_x\nabla_x^\top\left[\gamma_{k_s}(x,y)-\gamma_{k_1}(x,y)\right]+\frac{k_2^2-q}{k_1^2-k_2^2}\nu_x\nabla_x^\top\left[\gamma_{k_1}(x,y)-\gamma_{k_2}(x,y)\right]\nonumber\\
&&\quad+\partial_{\nu_x}\gamma_{k_s}(x,y)I+M(\partial_x,\nu_x)\left[2\mu E_{11}(x,y)-\gamma_{k_s}I\right],
\en
and
\be
\label{Ty}
&&T(\partial_y,\nu_y)E_{11}(x,y)\nonumber\\
&&=-\nu_y\nabla_y^\top\left[\gamma_{k_s}(x,y)-\gamma_{k_1}(x,y)\right]+\frac{k_2^2-q}{k_1^2-k_2^2}\nu_y\nabla_y^\top\left[\gamma_{k_1}(x,y)-\gamma_{k_2}(x,y)\right]\nonumber\\
&&\quad+\partial_{\nu_y}\gamma_{k_s}(x,y)I+M(\partial_y,\nu_y)\left[2\mu E_{11}(x,y)-\gamma_{k_s}I\right].
\en

With these identities, the proofs of Theorems~\ref{regK}-\ref{regN4} can be established.
\begin{proof}
The operator $K$ can be written as
\ben
K(U)(x)&=&  \int_\Gamma (\widetilde T^*(\partial _y,\nu _y)E(x,y))^\top U(y)ds_y \\
&=&\begin{bmatrix}
	K_1& K_2\\
	K_3 & K_4
\end{bmatrix}\begin{bmatrix}
	u\\
	p
\end{bmatrix}(x),
\enn
where
\ben
&&K_1(u)(x)=\int_\Gamma\left((T(\partial_y,\nu_y)E_{11}(x,y))^\top-\frac{\rho_f\omega^2\alpha}{\beta}E_{21}(x,y)\nu_y^\top\right)u(y)ds_y,\\
&&K_2(p)(x)=\int_\Gamma\left(-\beta E_{11}(x,y)\nu_y+\partial_{\nu_y}E_{21}(x,y)\right)p(y)ds_y,\\
&&K_3(u)(x)=\int_\Gamma \left((T(\partial_y,\nu_y)E_{12}(x,y))^\top-\frac{\rho_f\omega^2\alpha}{\beta}E_{22}(x,y)\nu_y^\top\right)u(y)ds_y,\\
&&K_4(p)(x)=\int_\Gamma\left(-\beta E_{12}^\top(x,y)\nu_y+\partial_{\nu_y}E_{22}(x,y)\right)p(y)ds_y.
\enn
Using (\ref{s4}) and (\ref{Ty}), we have
\be
\label{TE11T}
&&\quad\int_\Gamma(T(\partial_y,\nu_y)E_{11}(x,y))^\top u(y)ds_y\nonumber\\
&&=-\int_\Gamma\nabla_y\left[(\gamma_{k_s}(x,y)-\gamma_{k_1}(x,y))-\frac{k_2^2-q}{k_1^2-k_2^2}(\gamma_{k_1}(x,y)-\gamma_{k_2}(x,y))\right]\nu_y^\top u(y)ds_y\nonumber\\
&&\quad+\int_\Gamma\partial_{\nu_y}\gamma_{k_s}(x,y)Iu(y)ds_y+\int_\Gamma  \left[2\mu E_{11}(x,y) - \gamma _{k_s}(x,y)I\right] M(\partial_y,\nu_y)u(y)ds_y.
\en
Then we can obtain that
\ben
&&\quad K_1(u)(x)\\
&&=-\int_\Gamma\nabla_y\left[(\gamma_{k_s}(x,y)-\gamma_{k_1}(x,y))-\frac{k_2^2-q}{k_1^2-k_2^2}(\gamma_{k_1}(x,y)-\gamma_{k_2}(x,y))\right]\nu_y^\top u(y)ds_y\\
&&\quad+\int_\Gamma \left[\partial_{\nu _y}\gamma_{k_s}(x,y)I-\frac{\rho_f\omega^2\alpha}{\beta}E_{21}(x,y)\nu_y^\top\right]u(y)ds_y\\
&&\quad+\int_\Gamma\partial_{\nu_y}\gamma_{k_s}(x,y)Iu(y)ds_y+\int_\Gamma  \left[2\mu E_{11}(x,y) - \gamma _{k_s}(x,y)I\right] M(\partial_y,\nu_y)u(y)ds_y.
\enn
It follows from (\ref{eq.G1}) and (\ref{s3}) that
\be
&&\quad\int_\Gamma\partial_{\nu_y}E_{21}(x,y)p(y)ds_y\nonumber\\
&&=-\frac{\alpha-\beta}{(\lambda+2\mu)(k_1^2-k_2^2)}\int_\Gamma\partial_{\nu_y}\nabla_x(\gamma_{k_1}(x,y)-\gamma_{k_2}(x,y))p(y)ds_y\nonumber\\
&&=\frac{\alpha-\beta}{(\lambda+2\mu)(k_1^2-k_2^2)}\int_\Gamma\partial_{\nu_y}\nabla_y(\gamma_{k_1}(x,y)-\gamma_{k_2}(x,y))p(y)ds_y\nonumber\\
&&=\frac{\alpha-\beta}{(\lambda+2\mu)(k_1^2-k_2^2)}\int_\Gamma\left(M(\partial_y,\nu_y)\nabla_y+\nu_y\Delta_y\right)(\gamma_{k_1}(x,y)-\gamma_{k_2}(x,y))p(y)ds_y\nonumber\\
&&=-\frac{\alpha-\beta}{(\lambda+2\mu)(k_1^2-k_2^2)}\int_\Gamma(k_1^2\gamma_{k_1}(x,y)-k_2^2\gamma_{k_2}(x,y))\nu_yp(y)ds_y\nonumber\\
&&\quad-\frac{\alpha-\beta}{(\lambda+2\mu)(k_1^2-k_2^2)}\int_\Gamma M(\partial_y,\nu_y)p(y)\nabla_y(\gamma_{k_1}(x,y)-\gamma_{k_2}(x,y))ds_y\nonumber\\
&&=-\frac{\alpha-\beta}{(\lambda+2\mu)(k_1^2-k_2^2)}\int_\Gamma(k_1^2\gamma_{k_1}(x,y)-k_2^2\gamma_{k_2}(x,y))\nu_yp(y)ds_y\nonumber\\
&&\quad+\frac{\alpha-\beta}{(\lambda+2\mu)(k_1^2-k_2^2)}\left\{ {\int_\Gamma\nabla_y^\top(\gamma_{k_1}(x,y)-\gamma_{k_2}(x,y))M(\partial_y,\nu_y)p(y)}ds_y \right\}^\top,
\label{pyE21}
\en
which implies that
\ben
&&\quad K_2(p)(x)\\
&&=-\int_\Gamma  \left[\frac{\alpha-\beta }{(k_1^2 - k_2^2)(\lambda+2\mu)} (k_1^2\gamma _{k_1}(x,y) - k_2^2\gamma _{k_2}(x,y)) + \beta E_{11}(x,y)\right]\nu _yp(y)ds_y\\
&&\quad+\frac{\alpha-\beta}{(\lambda+2\mu)(k_1^2-k_2^2)}\left\{ {\int_\Gamma\nabla_y^\top(\gamma_{k_1}(x,y)-\gamma_{k_2}(x,y))M(\partial_y,\nu_y)p(y)}ds_y \right\}^\top.
\enn
From (\ref{eq.G2}) and (\ref{s2}), we can obtain that
\be
&&\quad\int_\Gamma(T(\partial_y,\nu_y)E_{12}(x,y))^\top u(y)ds_y\nonumber\\
&&=\frac{i\omega\gamma}{(\lambda+2\mu)(k_1^2-k_2^2)}\int_\Gamma(T(\partial_y,\nu_y)\nabla_x(\gamma_{k_1}(x,y)-\gamma_{k_2}(x,y)))^\top u(y)ds_y\nonumber\\
&&=-\frac{i\omega\gamma}{(\lambda+2\mu)(k_1^2-k_2^2)}\int_\Gamma(T(\partial_y,\nu_y)\nabla_y(\gamma_{k_1}(x,y)-\gamma_{k_2}(x,y)))^\top u(y)ds_y\nonumber\\
&&=-\frac{i\omega\gamma}{k_1^2-k_2^2}\int_\Gamma\Delta_y(\gamma_{k_1}(x,y)-\gamma_{k_2}(x,y))\nu_y^\top u(y)ds_y\nonumber\\
&&\quad-\frac{2i\mu\omega\gamma}{(\lambda+2\mu)(k_1^2-k_2^2)}\int_\Gamma (M(\partial_y,\nu_y)\nabla_y(\gamma_{k_1}(x,y)-\gamma_{k_2}(x,y)))^\top u(y)ds_y\nonumber\\
&&=\frac{i\omega\gamma}{k_1^2-k_2^2}\int_\Gamma(k_1^2\gamma_{k_1}(x,y)-k_2^2\gamma_{k_2}(x,y))\nu_y^\top u(y)ds_y\nonumber\\
&&\quad-\frac{2i\mu\omega\gamma}{(\lambda+2\mu)(k_1^2-k_2^2)}\int_\Gamma\nabla_y^\top(\gamma_{k_1}(x,y)-\gamma_{k_2}(x,y))M(\partial_y,\nu_y)u(y)ds_y,
\label{TyE12T}
\en
which yields
\ben
&\quad&K_3(u)(x)\\
&=&\int_\Gamma  \left[  \frac{i\omega\gamma}{(k_1^2-k_2^2)}(k_1^2\gamma _{k_1}(x,y) - k_2^2\gamma_{k_2}(x,y))-\frac{\rho_f\omega^2\alpha}{\beta}E_{22}\right]\nu _y^\top u(y)ds_y\\
&\quad&-\frac{2i\mu\omega\gamma}{(\lambda+2\mu)(k_1^2-k_2^2)}\int_\Gamma\nabla_y^\top(\gamma_{k_1}(x,y)-\gamma_{k_2}(x,y))M(\partial_y,\nu_y)u(y)ds_y.
\enn
The formula for $K_4(p)(x)$ can be obtained directly from its definition and this completes the proof of Theorem~{\ref{regK}}.

The operator $K'$ have the following form
\ben
K'(U)(x)= \begin{bmatrix}
	K_1'& K_2'\\
	K_3' & K_4'
\end{bmatrix}\begin{bmatrix}
	u\\
	p
\end{bmatrix}(x),\quad x\in\Gamma,
\enn
where the operators $K'_j,j=1,\cdots,4$ are denoted as
\ben
&&K'_1(u)(x)=\int_\Gamma  \left(T(\partial_x,\nu_x)E_{11}-\alpha \nu_x E_{12}^\top\right) u(y)ds_y, \\
&&K'_2(p)(x)=\int_\Gamma  \left(T(\partial_x,\nu_x)E_{21}-\alpha \nu_x E_{22}\right) p(y)ds_y, \\
&&K'_3(u)(x)=\int_\Gamma  \left(-\rho_f\omega^2\nu_x^\top E_{11}+ \partial_{\nu_x}E_{12}^\top\right)u(y)ds_y, \\
&&K'_4(p)(x)=\int_\Gamma  \left(-\rho_f\omega^2\nu_x^{\top}E_{21} +\partial_{\nu_x}E_{22}\right)p(y)ds_y.
\enn
From (\ref{Tx}), it can be obtained that
\ben
&&\quad K'_1(u)(x)\\
&&= -\int_\Gamma \nu _x\nabla _x^\top \left[(\gamma _{k_s}(x,y) - \gamma _{k_1}(x,y)) - \frac{k_2^2 - q}{k_1^2 - k_2^2}(\gamma _{k_1}(x,y) - \gamma _{k_2}(x,y))\right]u(y)ds_y\\
&&\quad + \int_\Gamma  \left[\partial _{\nu _x}\gamma _{k_s}(x,y)I - \alpha \nu _xE_{12}^\top(x,y) +M(\partial_x,\nu_x)(2\mu E_{11}(x,y) - \gamma _{k_s}(x,y)I) \right]u(y)ds_y.
\enn
It follows from (\ref{eq.G2}) that
\be
&&\quad \int_\Gamma T(\partial_x,\nu_x)E_{21}(x,y)p(y)ds_y\nonumber\\
&&=-\frac{\alpha-\beta}{(\lambda+2\mu)(k_1^2-k_2^2)}\int_\Gamma T(\partial_x,\nu_x)\nabla_x (\gamma_{k_1}(x,y)-\gamma_{k_2}(x,y))p(y)ds_y\nonumber\\
&&=-\frac{\alpha-\beta}{(k_1^2-k_2^2)}\int_\Gamma\nu_x \Delta_x (\gamma_{k_1}(x,y)-\gamma_{k_2}(x,y))p(y)ds_y\nonumber\\
&&\quad -\frac{2\mu(\alpha-\beta)}{(\lambda+2\mu)(k_1^2-k_2^2)} \int_\Gamma M(\partial_x ,\nu_x)\nabla_x(\gamma_{k_1}(x,y)-\gamma_{k_2}(x,y))p(y)ds_y\nonumber\\	&&=\frac{\alpha-\beta}{(k_1^2-k_2^2)}\int_\Gamma(k_1^2\gamma_{k_1}(x,y) -k_2^2\gamma_{k_2}(x,y))\nu_xp(y)ds_y\nonumber\\
&&\quad -\frac{2\mu(\alpha-\beta)}{(\lambda+2\mu)(k_1^2-k_2^2)} \int_\Gamma M(\partial_x,\nu_x)\nabla_x(\gamma_{k_1}(x,y)-\gamma_{k_2}(x,y))p(y)ds_y.
\label{TE21}
\en
Therefore,
\ben
&&K'_2(p)(x)\\
&&=\int_\Gamma  \left(\frac{\alpha-\beta}{k_1^2 - k_2^2} (k_1^2\gamma _{k_1}(x,y)-k_2^2\gamma _{k_2}(x,y)) - \alpha E_{22}(x,y)\right)\nu_xp(y)ds_y\\
&&\quad-\frac{2\mu(\alpha-\beta)}{(\lambda+2\mu)(k_1^2-k_2^2)}\int_\Gamma M(\partial_x,\nu_x)\nabla_x(\gamma_{k_1}(x,y)-\gamma_{k_2}(x,y))p(y)ds_y.
\enn
On the other hand, we obtain from (\ref{eq.G1}) that
\be
&&\quad\int_\Gamma\partial_{\nu_x}E_{12}^\top(x,y)u(y)ds_y\nonumber\\
&&=\frac{i\omega\gamma}{(\lambda+2\mu)(k_1^2-k_2^2)}\int_\Gamma\partial_{\nu_x}\nabla_x^\top\left(\gamma_{k_1}(x,y)-\gamma_{k_2}(x,y)\right)u(y)ds_y\nonumber\\
&&=\frac{i\omega \gamma}{(\lambda+2\mu)(k_1^2-k_2^2)}\int_\Gamma\nu_x^\top \Delta_x\left(\gamma_{k_1}(x,y)-\gamma_{k_2}(x,y)\right)u(y)ds_y\nonumber\\
&&\quad+\frac{i\omega \gamma}{(\lambda+2\mu)(k_1^2-k_2^2)}\int_\Gamma\left\{M(\partial_x,\nu_x)\nabla_x \left(\gamma_{k_1}(x,y)-\gamma_{k_2}(x,y)\right) \right\}^\top u(y)ds_y\nonumber\\
&&=-\frac{i\omega\gamma}{(\lambda+2\mu)(k_1^2-k_2^2)}\int_\Gamma\left(k_1^2\gamma_{k_1}(x,y)-k_2^2\gamma_{k_2}(x,y)\right)\nu_x^\top u(y)ds_y\nonumber\\
&&\quad+\frac{i\omega\gamma}{(\lambda+2\mu)(k_1^2-k_2^2)}\int_\Gamma\left\{M(\partial_x,\nu_x)\nabla_x \left(\gamma_{k_1}(x,y)-\gamma_{k_2}(x,y)\right) \right\}^\top u(y)ds_y.
\label{parE12}
\en
Letting $R_1(x,y)=\left(\gamma_{k_1}(x,y)-\gamma_{k_2}(x,y)\right)$, note that
\ben
&&\quad\left\{M(\partial_x,\nu_x)\nabla_xR_1(x,y) \right\}^\top u(y)\\
&&=\left\{ \begin{bmatrix}
	m_x^{11}&m_x^{12}&m_x^{13}\\
	m_x^{21}&m_x^{22}&m_x^{23}\\
	m_x^{31}&m_x^{32}&m_x^{33}
\end{bmatrix}\left( {\begin{array}{*{20}{c}}
	\partial_{x_1}R_1(x,y)\\
	\partial_{x_2}R_1(x,y)\\
	\partial_{x_3}R_1(x,y)
\end{array}} \right) \right\}^\top \left( {\begin{array}{*{20}{c}}
u_1(y)\\
u_2(y)\\
u_3(y)
\end{array}} \right)\\
&&=\sum\limits_{i,j=1}^{3}m_x^{ij}\partial_{x_j}\left(\gamma_{k_1}(x,y)-\gamma_{k_2}(x,y)\right)u_i(y)\\
&&=\begin{bmatrix}
	m_x^{11}&m_x^{12}&m_x^{13}\\
	m_x^{21}&m_x^{22}&m_x^{23}\\
	m_x^{31}&m_x^{32}&m_x^{33}
\end{bmatrix}:\left( {\begin{array}{*{20}{c}}
	\partial_{x_1}R_1(x,y)u_1(y)&\partial_{x_2}R_1(x,y)u_1(y)&\partial_{x_3}R_1(x,y)u_1(y)\\
	\partial_{x_2}R_1(x,y)u_2(y)&\partial_{x_2}R_1(x,y)u_2(y)&\partial_{x_3}R_1(x,y)u_2(y)\\
	\partial_{x_3}R_1(x,y)u_3(y)&\partial_{x_2}R_1(x,y)u_3(y)&\partial_{x_3}R_1(x,y)u_3(y)
\end{array}} \right)\\
&&=M(\partial_x,\nu_x):\left(u(y)\nabla_x^\top R_1(x,y)\right).
\enn
Then we have
\ben
&&\quad K^{\prime}_3(u)(x)\\
&&= -\int_\Gamma\left[\rho_f\omega^2\nu_x^\top E_{11}(x,y) +\frac{i\omega\gamma}{(\lambda+2\mu)(k_1^2-k_2^2)}(k_1^2\gamma _{k_1}(x,y) - k_2^2\gamma _{k_2}(x,y))\nu _x^\top\right] u(y)ds_y\\
&&\quad+\frac{i\omega\gamma}{(k_1^2-k_2^2)(\lambda+2\mu)}M(\partial_x,\nu_x):\int_\Gamma u(y)\nabla_x^\top\left[\gamma _{k_1}(x,y) - \gamma_{k_2}(x,y))\right] ds_y.\\
\enn
Using the definition, the formula for $K_4'$ can be obtained directly. Then the Theorem~(\ref{regKP}) can be proved.

Now we investigate the hyper-singular operator $N$. Note that the hyper-singular operator of $N$ can be written as
\ben
N(\psi)(x)=\begin{bmatrix}
	N_1 & N_2\\
	N_3 & N_4
\end{bmatrix}\begin{bmatrix}
	u\\
	p
\end{bmatrix}(x),\quad x\in\Gamma,
\enn
where
\ben
N_1(u)(x) &=&\int_\Gamma\left[T(\partial _x,\nu _x)(T(\partial _y,\nu _y)E_{11}(x,y))^\top - \frac{\rho_f\omega^2\alpha}{\beta} T(\partial _x,\nu _x)E_{21}(x,y)\nu _y^\top \right]u(y)ds_y\\
&&-\int_\Gamma\left[\alpha\nu _x(T(\partial _y,\nu _y)E_{12}(x,y))^\top - \frac{\rho_f\omega^2\alpha^2}{\beta}E_{22}(x,y)\nu _x\nu _y^\top \right]u(y)ds_y,\\
N_2(p)(x) &=&-\int_\Gamma\left[\beta T(\partial _x,\nu _x)E_{11}(x,y)\nu _y - \partial_{\nu _y}T(\partial _x,\nu _x)E_{21}(x,y) \right]p(y)ds_y\\
&&+\int_\Gamma\left[\alpha\beta\nu_xE_{12}^\top\nu _y-\alpha\partial_{\nu_y}E_{22}(x,y)\nu_x \right]p(y)ds_y,\\
N_3(u)(x)  &=& -\int_\Gamma\left[\rho_f\omega^2\nu_x^\top(T(\partial _y,\nu_y)E_{11}(x,y))^\top-\frac{\rho_f^2\omega^4\alpha}{\beta}\nu_x^\top E_{21}\nu_y^\top \right]u(y)ds_y\\
&&+\int_\Gamma\left[\partial_{\nu_x}(T(\partial_y,\nu_y)E_{12}(x,y))^\top-\frac{\rho_f\omega^2\alpha}{\beta}\partial_{\nu_x}E_{22}(x,y)\nu_y^\top\right]u(y)ds_y,\\
N_4(p)(x) &=&\int_\Gamma\left[\rho_f\omega^2\beta\nu _x^\top E_{11}(x,y)\nu_y -\rho_f\omega^2\nu _x^\top\partial_{\nu _y}E_{21}(x,y) \right]p(y)ds_y\\
&&-\int_\Gamma\left[\beta\partial_{\nu_x}E_{12}^\top(x,y)\nu_y-\partial_{\nu_x}(\partial_{\nu_y}E_{22}(x,y)) \right]p(y)ds_y.
\enn

Considering the following term
\ben
T(\partial_x,\nu_x)\int_\Gamma(T(\partial_y,\nu_y)E_{11}(x,y))^\top u(y)ds_y,
\enn
we first set that
\ben
&&f_1(x)=\int_\Gamma\nabla_y\left(\gamma_{k_s}(x,y)-\gamma_{k_1}(x,y)\right)\nu_y^\top u(y)ds_y,\\
&&f_2(x)=\int_\Gamma\nabla_y\left(\gamma_{k_1}(x,y)-\gamma_{k_2}(x,y)\right)\nu_y^\top u(y)ds_y,\\
&&f_3(x)=\int_\Gamma\partial_{\nu_y}\gamma_{k_s}(x,y) u(y)ds_y,\\
&&f_4(x)=\int_\Gamma\left(2\mu E_{11}(x,y)-\gamma_{k_s}(x,y)I\right)M(\partial_y,\nu_y)u(y)ds_y,
\enn
and
\ben
g_i(x)=\mu\partial_{\nu_x}f_i(x) + (\lambda + \mu )\nu_x\nabla_x\cdot f_i(x) + \mu M(\pa_x,\nu_x)f_i(x).
\enn
Thus, we can obtain from (\ref{eq.G1}) and (\ref{eq.G2}) that
\be
g_1(x) &=&(\lambda+2\mu)\int_\Gamma\left[k_s^2\gamma_{k_s}(x,y)-k_1^2\gamma_{k_1}(x,y)\right]\nu_x\nu _y^\top u(y)d{s_y}\nonumber\\
&& + 2\mu \int_\Gamma M(\partial_x,\nu_x)\nabla_y\left[\gamma_{k_s}(x,y) - \gamma _{k_1}(x,y)\right]\nu _y^\top u(y)ds_y,
\label{eq.g1}
\en
and
\be
g_2(x) &=&(\lambda+2\mu)\int_\Gamma\left[k_1^2\gamma_{k_1}(x,y)-k_2^2\gamma_{k_2}(x,y)\right]\nu_x\nu _y^\top u(y)ds_y\nonumber\\
&&+2\mu \int_\Gamma M(\partial_x,\nu_x)\nabla_y\left[\gamma _{k_1}(x,y)-\gamma_{k_2}(x,y)\right]\nu _y^\top u(y)ds_y.
\label{eq.g2}
\en
Relying on the results of the Helmholtz equation, we have
\be
&&\quad\int_\Gamma\partial_{\nu_x}\partial_{\nu_y}\gamma_{k_s}(x,y)u(y)ds_y\nonumber\\
&&=\int_\Gamma(\nu_x\times \nabla_x\gamma_{k_s}(x,y))\cdot (\nu_y\times \nabla_y u)(x)ds_y+k_s^2\int_\Gamma\gamma_{k_s}(x,y)\nu_x^\top\nu_yu(y)ds_y
\label{ppgks}
\en
Therefore, $g_3(x)$ can be expressed as
\be
g_3(x)&=&\mu\int_\Gamma \partial_{\nu _x}\partial_{\nu _y}\gamma_{k_s}(x,y)u(y) ds_y+(\lambda+\mu )\int_\Gamma\nu_x\nabla_x^\top\partial_{\nu_y}\gamma _{k_s}(x,y)u(y)ds_y \nonumber\\
&&+ \mu \int_\Gamma M(\partial_x,\nu_x)\partial_{\nu _y}\gamma _{k_s}(x,y)u(y)ds_y \nonumber\\
&=&\mu \int_\Gamma(\nu_x\times \nabla_x\gamma_{k_s}(x,y))\cdot (\nu_y\times \nabla_yu(y)) ds_y + \mu k_s^2\int_\Gamma\gamma_{k_s}(x,y)\nu_x^\top\nu_yu(y)ds_y\nonumber\\
&&+(\lambda+\mu )\int_\Gamma\nu_x\nabla_x^\top\partial_{\nu_y}\gamma _{k_s}(x,y)u(y)ds_y + \mu \int_\Gamma M(\partial_x,\nu_x)\partial_{\nu _y}\gamma _{k_s}(x,y)u(y)ds_y.
\label{eq.g3}
\en
For $g_4(x)$, we know from (\ref{Tx}) that
\be
g_4(x)  &=& \mu\int_\Gamma\nu_x^\top\nabla_x\gamma_{k_s}(x,y) M(\partial_y,\nu_y)ds_y\nonumber\\
&&-2\mu\int_\Gamma\nu_x\nabla_x^\top\left[\gamma_{k_s}(x,y)
-\gamma_{k_1}(x,y)\right] M(\partial_y,\nu_y)u(y)ds_y\nonumber\\
&&+\frac{2\mu(k_1^2-q)}{k_1^2-k_2^2}\int_\Gamma\nu_x\nabla_x^T\left[\gamma _{k_1}(x,y) - \gamma_{k_2}(x,y)\right] M(\partial_y,\nu_y)u(y)ds_y\nonumber\\
&& + 4\mu^2\int_\Gamma M(\partial_x,\nu_x) E_{11}(x,y) M(\partial_y,\nu_y)u(y)ds_y\nonumber\\
&& - 3\mu  \int_\Gamma M(\partial_x,\nu_x)\gamma_{k_s}(x,y)M(\partial_y,\nu_y)u(y)ds_y\nonumber\\
&& - (\lambda  + \mu )\int_\Gamma\nu_x\nabla_x^\top\gamma_{k_s}(x,y)M(\partial_y,\nu_y)u(y)ds_y.
\label{eq.g4}
\en
Therefore, (\ref{eq.g1})-(\ref{eq.g4}) yields
\be
&&\quad T(\partial_x,\nu_x)\int_\Gamma(T(\partial_y,\nu_y)E_{11}(x,y))^\top u(y)ds_y\nonumber\\
&&=-g_1(x)+\frac{k_2^2-q}{k_1^2-k_2^2}g_2(x)+g_3(x)+g_4(x)\nonumber\\
&&=-(\rho-\beta\rho_f)\omega^2\int_\Gamma\gamma_{k_s}(x,y)(\nu_x\nu_y^\top - \nu_x^\top\nu_yI )u(y)ds_y\nonumber\\
&&\quad+\int_\Gamma\left[D_1\gamma_{k_1}(x,y) - D_2\gamma_{k_2}(x,y)\right] \nu_x\nu _y^\top u(y)ds_y\nonumber\\
&&\quad+ \mu\int_\Gamma(\nu_x\times \nabla_x\gamma_{k_s}(x,y))\cdot (\nu_y\times \nabla_yu(y))ds_y\nonumber\\
&&\quad+ 4\mu^2\int_\Gamma M(\partial_x,\nu_x) E_{11}(x,y) M(\partial_y,\nu_y)u(y)ds_y\nonumber\\
&&\quad- 3\mu  \int_\Gamma M(\partial_x,\nu_x)\gamma_{k_s}(x,y)M(\partial_y,\nu_y)u(y)ds_y\nonumber\\
&&\quad-2\mu\int_\Gamma\nu_x\nabla_x^\top\left[\gamma_{k_s}(x,y) - \gamma _{k_1}(x,y)\right] M(\partial_y,\nu_y)u(y)ds_y\nonumber\\
&&\quad- 2\mu \int_\Gamma M(\partial_x,\nu_x)\nabla_y\left[\gamma_{k_s}(x,y) - \gamma _{k_1}(x,y)\right] \nu _y^\top u(y)ds_y\nonumber\\
&&\quad+\frac{2\mu(k_2^2 - q)}{k_1^2 - k_2^2}\int_\Gamma \nu_x\nabla _x^\top\left[\gamma_{k_1}(x,y) - \gamma_{k_2}(x,y)\right] M(\partial_y,\nu_y)u(y)ds_y\nonumber
\en
\be
&&\quad+ \frac{2\mu(k_2^2 - q)}{k_1^2 - k_2^2}\int_\Gamma  M(\partial_x,\nu_x)\nabla _y\left[\gamma_{k_1}(x,y) - \gamma_{k_2}(x,y)\right]\nu_y^\top u(y)ds_y\nonumber\\
&&\quad+\mu \int_\Gamma M(\partial_x,\nu_x)\partial_{\nu _y}\gamma _{k_s}(x,y)u(y)ds_y+\mu\int_\Gamma\partial_{\nu _x}\gamma_{k_s}(x,y) M(\partial_y,\nu_y)ds_y,
\label{eq.N11}
\en
with
\be
\nonumber
D_1=\frac{k_1^2(\lambda + 2\mu )(k_1^2 - q)}{k_1^2 - k_2^2},\quad D_2=\frac{k_2^2(\lambda  + 2\mu )(k_2^2 - q)}{k_1^2 - k_2^2}.
\en
On the other hand, we obtain from (\ref{eq.G2}) that
\be
&&\quad\int_\Gamma T(\partial_x,\nu_x)E_{21}(x,y)\nu_y^\top u(y)ds_y\nonumber\\
&&= \frac{\alpha-\beta}{k_1^2 - k_2^2}\int_\Gamma\left[k_1^2\gamma_{k_1}(x,y)- k_2^2\gamma_{k_2}(x,y)\right] \nu_x\nu_y^\top u(y)ds_y\nonumber\\
&&\quad+\frac{2\mu(\alpha-\beta)}{(\lambda  + 2\mu )(k_1^2 - k_2^2)}\int_\Gamma M(\partial_x,\nu_x) \nabla_y\left[\gamma_{k_1}(x,y) - \gamma _{k_2}(x,y)\right]\nu_y^\top u(y)ds_y,
\label{eq.N12}
\en
and
\be
&&\quad\int_\Gamma\nu _x(T(\partial_y,\nu_y)E_{12}(x,y))^\top u(y)ds_y\nonumber\\
&&= \frac{i\omega\gamma}{k_1^2 - k_2^2}\int_\Gamma\left[k_1^2\gamma_{k_1}(x,y) - k_2^2\gamma_{k_2}(x,y)\right]\nu_x\nu_y^\top u(y)ds_y\nonumber\\
&&\quad+\frac{2i\omega\mu\gamma}{(\lambda  + 2\mu )(k_1^2 - k_2^2)}\int_\Gamma\nu_x\nabla_x^\top(\gamma_{k_1}(x,y) - \gamma _{k_2}(x,y)) M(\partial_y,\nu_y)ds_y.
\label{eq.N13}
\en
Combining (\ref{eq.N11})-(\ref{eq.N13}), we have
\ben
&&\quad N_1(u)(x)\\
&&=-(\rho-\beta\rho_f) \omega^2\int _\Gamma \gamma_{k_s}\left(x,y\right)\left(\nu_x\nu_y^\top -\nu^\top_x\nu_yI\right)u(y)ds_{y}\\
&&\quad+\int_\Gamma\left[C_{1}\gamma_{k_{1}}(x,y)-C_{2}\gamma_{k_{2}}(x,y)\right] \nu_{x}\nu_{y}^{\top}u(y)ds_{y},\\	
&&\quad+\int_{\Gamma}M(\partial_x,\nu_x)\left[4\mu^2E_{11}(x,y)-3\mu\gamma_{k_s}(x,y)I\right]M(\partial_y,\nu_y)u(y)ds_y\\
&&\quad+\mu\int_\Gamma\tau_2 \gamma_{k_s}(x,y)\tau_1u(y)ds_y\\
&&\quad+\int_\Gamma M(\partial_x,\nu_x)\nabla _y \left[-2\mu(\gamma_{k_s}(x,y)-\gamma_{k_1}(x,y) )+C_3(\gamma_{k_1}(x,y)-\gamma_{k_2}(x,y) )\right] \nu^\top_y u(y)ds_y\\
&&\quad+\mu\int_\Gamma M(\partial_x,\nu_x)\partial_{\nu_y}\gamma_{k_s}(x,y) u(y)ds_y\\
&&\quad+\int_\Gamma \nu_x\nabla _x^\top \left[-2\mu(\gamma_{k_s}(x,y)-\gamma_{k_1}(x,y) )+C_4(\gamma_{k_1}(x,y)-\gamma_{k_2}(x,y) )\right]M(\partial_y,\nu_y)u(y)ds_y\\
&&\quad+\mu\int_\Gamma\partial_{\nu_x}\gamma_{k_s}(x,y)M(\partial_y,\nu_y) u(y)ds_y.
\enn

For $N_2$, it follows from (\ref{eq.G1}) and (\ref{eq.G2}) that
\be
&&\quad \int_\Gamma \partial_{\nu_y}T(\partial_x,\nu_x)E_{21}(x,y)p(y)ds_y \nonumber\\
&&=-\frac{\alpha-\beta}{(\lambda+2\mu)(k^2_1-k^2_2)}\int_\Gamma \partial_{\nu_y}T(\partial_x,\nu_x)\nabla_x\left[\gamma_{k_1}(x,y)-\gamma_{k_2}(x,y)\right] p(y)ds_y\nonumber \\
&&=\frac{\alpha-\beta}{k^2_1-k^2_2}\int_\Gamma \partial_{\nu_y} \left[k_1^2\gamma_{k_1}(x,y)-k_2^2\gamma_{k_2}(x,y)\right]\nu_x p(y)ds_y \nonumber\\
&&\quad+\frac{2\mu(\alpha-\beta)}{(\lambda+2\mu)(k^2_1-k^2_2)}\int_\Gamma M(\partial_x,\nu_x)\partial_{\nu_y} \nabla _y\left[ \gamma_{k_1}(x,y) - \gamma _{k_2}(x,y) \right] p(y)ds_y \nonumber\\
&&=\frac{\alpha-\beta}{k^2_1-k^2_2}\int_\Gamma \partial_{\nu_y} \left[k_1^2\gamma_{k_1}(x,y)-k_2^2\gamma_{k_2}(x,y)\right]\nu_x p(y)ds_y \nonumber\\
&&\quad-\frac{2\mu(\alpha-\beta)}{(\lambda+2\mu)(k^2_1-k^2_2)}\int_\Gamma M(\partial_x,\nu_x)\left[k_1^2\gamma_{k_1}(x,y)-k_2^2\gamma_{k_2}(x,y)\right]\nu_y p(y)ds_y\nonumber \\
&&\quad+\frac{2\mu(\alpha-\beta)}{(\lambda+2\mu)(k^2_1-k^2_2)}\int_\Gamma M(\partial_x,\nu_x)M(\partial_y,\nu_y)\nabla _y\left[ \gamma_{k_1}(x,y) - \gamma _{k_2}(x,y) \right] p(y)ds_y\nonumber\\
&&=\frac{\alpha-\beta}{k^2_1-k^2_2}\int_\Gamma \partial_{\nu_y} \left[k_1^2\gamma_{k_1}(x,y)-k_2^2\gamma_{k_2}(x,y)\right]\nu_x p(y)ds_y\nonumber \\
&&\quad-\frac{2\mu(\alpha-\beta)}{(\lambda+2\mu)(k^2_1-k^2_2)}\int_\Gamma M(\partial_x,\nu_x)\left[k_1^2\gamma_{k_1}(x,y)-k_2^2\gamma_{k_2}(x,y)\right]\nu_y p(y)ds_y\nonumber \\
&&\quad+\frac{2\mu(\alpha-\beta)}{(\lambda+2\mu)(k^2_1-k^2_2)}\int_\Gamma M(\partial_x,\nu_x)\left\{ \nabla _y^\top\left[ \gamma_{k_1}(x,y) - \gamma _{k_2}(x,y) \right]M(\partial_y,\nu_y)p(y) \right\}^\top ds_y,
\en
which, in corporation with (\ref{Tx}), yields that
\ben
N_2(p)(x) &=&\beta \int_\Gamma  \left[\nu _x\nabla _x^\top(\gamma _{k_s}(x,y) - \gamma _{k_1}(x,y)) - \partial _{\nu _x}\gamma _{k_s}(x,y)I\right]\nu _yp(y) ds_y\\
&& +(\frac{i\omega\gamma\alpha\beta}{(k_1^2-k_2^2)(\lambda+2\mu)}-\frac{\beta(k_2^2-q)}{k_1^2-k_2^2})\int_\Gamma \nu _x\nabla _x^\top(\gamma _{k_1}(x,y) - \gamma _{k_2}(x,y))\nu _yp(y) ds_y\\
&& + \frac{\alpha-\beta }{k_1^2 - k_2^2}\int_\Gamma  \partial _{\nu _y}\left[(k_1^2\gamma _{k_1}(x,y) - k_2^2\gamma _{k_2}(x,y))\right] \nu _xp(y)ds_y\\
&& +\frac{\alpha }{k_1^2 - k_2^2}\int_\Gamma  \partial _{\nu _y}\left[(k_p^2 - k_1^2)\gamma _{k_1}(x,y)-(k_p^2 - k_2^2)\gamma _{k_2}(x,y) \right]\nu _xp(y)ds_y\\
&& +\frac{2\mu(\alpha-\beta)}{(\lambda+2\mu)(k^2_1-k^2_2)}\int_\Gamma M(\partial_x,\nu_x)\left\{ \nabla _y^\top\left[ \gamma_{k_1}(x,y) - \gamma _{k_2}(x,y) \right]M(\partial_y,\nu_y)p(y) \right\}^\top ds_y\\
&& - \beta \int_\Gamma M(\partial_x,\nu_x) \left[2\mu E_{11}(x,y)- \gamma _{k_s}(x,y)I)\right]\nu _yp(y)ds_y\\
&& -\frac{2\mu(\alpha-\beta)}{(\lambda+2\mu)(k^2_1-k^2_2)}\int_\Gamma M(\partial_x,\nu_x)\left[k_1^2\gamma_{k_1}(x,y)-k_2^2\gamma_{k_2}(x,y)\right]\nu_y p(y)ds_y.
\enn

For $N_3$, we mainly need to consider the following term
\be
&&\quad \int_\Gamma \partial_{\nu_x}\left(T(\partial_y,\nu_y)E_{12}(x,y)\right)^\top u(y)ds_y\nonumber \\
&&=\frac{i\omega\gamma}{k^2_1-k^2_2}\int_\Gamma \partial_{\nu _x}\left[k_1^2\gamma_{k_1}(x,y)-k_2^2\gamma_{k_2}(x,y)\right]\nu_y^\top u(y)ds_y\nonumber\\
&&\quad-\frac{2i\mu\omega\gamma}{(k^2_1-k^2_2)(\lambda+2\mu)}\int_\Gamma \left[k_1^2\gamma_{k_1}(x,y)-k_2^2\gamma_{k_2}(x,y)\right]\nu_x^\top M(\partial_y,\nu_y) u(y)ds_y\nonumber\\
&&\quad+\frac{2i\mu\omega\gamma}{(k^2_1-k^2_2)(\lambda+2\mu)}\int_\Gamma \left\{M(\partial_x,\nu_x)\nabla_x\left[\gamma_{k_1}(x,y)-\gamma_{k_2}(x,y)\right] \right\}^\top M(\partial_y,\nu_y) u(y)dsy\nonumber
\en
\be
&&=\frac{i\omega\gamma}{k^2_1-k^2_2}\int_\Gamma \partial_{\nu _x}\left[k_1^2\gamma_{k_1}(x,y)-k_2^2\gamma_{k_2}(x,y)\right]\nu_y^\top u(y)ds_y\nonumber\\
&&\quad-\frac{2i\mu\omega\gamma}{(k^2_1-k^2_2)(\lambda+2\mu)}\int_\Gamma \left[k_1^2\gamma_{k_1}(x,y)-k_2^2\gamma_{k_2}(x,y)\right]\nu_x^\top M(\partial_y,\nu_y) u(y)ds_y\nonumber\\
&&\quad+\frac{2i\mu\omega\gamma}{(k^2_1-k^2_2)(\lambda+2\mu)}M(\partial_x,\nu_x):\int_\Gamma M(\partial_y,\nu_y) u(y)\nabla_x^\top\left[\gamma_{k_1}(x,y)-\gamma_{k_2}(x,y)\right] dsy.
\label{parTE12T}
\en
By a combination of (\ref{TE11T}) and (\ref{TyE12T}), we can obtain that
\ben
&&\quad N_3(u)(x)\\
&&=-\rho_f\omega^2 \int_\Gamma \left[\partial_{\nu_x}\left(\gamma _{k_s}(x,y) - \gamma _{k_1}(x,y)\right)\nu_y^\top+\partial_{\nu _y}\gamma_{k_s}(x,y)\nu_x^\top \right]u(y)ds_y\\
&&\quad+(\frac{\rho_f\omega^2(k_2^2-q)}{k_1^2-k_2^2}-\frac{\rho_f^2\omega^4\alpha(\alpha-\beta)}{\beta(\lambda+2\mu)(k_1^2-k_2^2)})\int_\Gamma \partial_{\nu_x}\left[\gamma _{k_1}(x,y) - \gamma _{k_2}(x,y)\right]\nu_y^\top u(y)ds_y\\
&&\quad+\frac{i\omega\gamma}{k_1^2-k_2^2}\int_\Gamma \partial_{\nu_x}\left[k_1^2\gamma _{k_1}(x,y) - k_2^2\gamma _{k_2}(x,y)\right]\nu_y^\top u(y)ds_y\\
&&\quad+\frac{\rho_f\omega^2\alpha}{\beta(k_1^2-k_2^2)}\int_\Gamma \partial_{\nu_x}\left[(k_p^2-k_1^2)\gamma _{k_1}(x,y) - (k_p^2-k_2^2)\gamma _{k_2}(x,y)\right]\nu_y^\top u(y)ds_y\\
&&\quad-\rho_f\omega^2\int_\Gamma  \nu_x^\top\left[2\mu E_{11}(x,y) - \gamma _{k_s}(x,y)I\right]u(y) ds_y\\
&&\quad- \frac{2i\mu \omega \gamma }{(\lambda  + 2\mu )(k_1^2 - k_2^2)}\int_\Gamma  \left[k_1^2\gamma _{k_1}(x,y) - k_2^2\gamma _{k_2}(x,y)\right]\nu _x^\top u(y) d{s_y}\\
&&\quad+\frac{2i\mu\omega\gamma}{(k^2_1-k^2_2)(\lambda+2\mu)}M(\partial_x,\nu_x):\int_\Gamma M(\partial_y,\nu_y) u(y)\nabla_x^\top\left[\gamma_{k_1}(x,y)-\gamma_{k_2}(x,y)\right] dsy.
\enn

Using (\ref{pyE21}), we have
\ben
&&\quad-\rho_f\omega^2\int_\Gamma\nu_x^\top\partial_{\nu_y}E_{21}(x,y)p(y)ds_y\\
&&=\frac{\rho_f\omega^2(\alpha-\beta)}{(\lambda+2\mu)(k_1^2-k_2^2)}\int_\Gamma\partial_{\nu_x}\partial_{\nu_y}\left(\gamma_{k_1}(x,y)-\gamma_{k_2}(x,y)\right)p(y)ds_y\\
&&=\frac{\rho_f\omega^2(\alpha-\beta)}{(\lambda+2\mu)(k_1^2-k_2^2)}\int_\Gamma\left[k_1^2\gamma_{k_1}(x,y)-k_2^2\gamma_{k_2}(x,y)\right]\nu_x^\top\nu_yp(y)ds_y\\
&&\quad+\frac{\rho_f\omega^2(\alpha-\beta)}{(\lambda+2\mu)(k_1^2-k_2^2)}\int_\Gamma\left(\nu_x\times\nabla_x\left[\gamma_{k_1}(x,y)-\gamma_{k_2}(x,y)\right]\right)\cdot\left( \nu_y\times\nabla_yp(y)\right)ds_y\\
&&=\frac{\rho_f\omega^2(\alpha-\beta)}{(\lambda+2\mu)(k_1^2-k_2^2)}\int_\Gamma\left[k_1^2\gamma_{k_1}(x,y)-k_2^2\gamma_{k_2}(x,y)\right]\nu_x^\top\nu_yp(y)ds_y\\
&&\quad+\frac{\rho_f\omega^2(\alpha-\beta)}{(\lambda+2\mu)(k_1^2-k_2^2)}\int_\Gamma\tau_2\left[\gamma_{k_1}(x,y)-\gamma_{k_2}(x,y)\right]\tau_1p(y)ds_y.
\enn
Due to (\ref{parE12}), we can obtain that
\ben
&&\quad-\beta\int_\Gamma\partial_{\nu_x}E_{12}^\top\nu_yp(y)ds_y\\
&&=\frac{i\omega\beta\gamma}{(\lambda+2\mu)(k_1^2-k_2^2)}\int_\Gamma\partial_{\nu_x}\partial_{\nu_y}\left(\gamma_{k_1}(x,y)-\gamma_{k_2}(x,y)\right)p(y)ds_y\\
&&=\frac{i\omega\beta\gamma}{(\lambda+2\mu)(k_1^2-k_2^2)}\int_\Gamma\left[k_1^2\gamma_{k_1}(x,y)-k_2^2\gamma_{k_2}(x,y)\right]\nu_x^\top\nu_yp(y)ds_y\\
&&\quad+\frac{i\omega\beta\gamma}{(\lambda+2\mu)(k_1^2-k_2^2)}\int_\Gamma\left(\nu_x\times\nabla_x\left[\gamma_{k_1}(x,y)-\gamma_{k_2}(x,y)\right]\right)\cdot\left( \nu_y\times\nabla_yp(y)\right)ds_y\\
&&=\frac{i\omega\beta\gamma}{(\lambda+2\mu)(k_1^2-k_2^2)}\int_\Gamma\left[k_1^2\gamma_{k_1}(x,y)-k_2^2\gamma_{k_2}(x,y)\right]\nu_x^\top\nu_yp(y)ds_y\\
&&\quad+\frac{i\omega\beta\gamma}{(\lambda+2\mu)(k_1^2-k_2^2)}\int_\Gamma\tau_2\left[\gamma_{k_1}(x,y)-\gamma_{k_2}(x,y)\right]\tau_1p(y)ds_y.
\enn
Following the result (\ref{ppgks}), we have
\ben
&&\quad\int_\Gamma \partial_{\nu _x} \partial_{\nu _y}E_{22}(x,y)p(y)ds_y\\
&&=-\frac{1}{k_1^2-k_2^2}\int_\Gamma \partial_{\nu _x} \partial_{\nu _y}\left[(k^2_p-k^2_1)\gamma_{k_1}(x,y)-(k^2_p-k^2_2)\gamma_{k_2}(x,y)\right]p(y)ds_y\\
&&=-\frac{1}{k^2_1-k^2_2}\int_\Gamma  \left[(k^2_p-k^2_1)k^2_1\gamma_{k_1}(x,y)-(k^2_p-k^2_2)k^2_2\gamma_{k_2}(x,y)\right]\nu_x^\top\nu_y p(y)ds_y\\
&&\quad-\frac{1}{k^2_1-k^2_2}\int_\Gamma\left(\nu_x\times\nabla_x\left[(k_p^2-k_1^2)\gamma_{k_1}(x,y)-(k_p^2-k_2^2)\gamma_{k_2}(x,y)\right]\right)\cdot\left( \nu_y\times\nabla_yp(y)\right)ds_y\\
&&=-\frac{1}{k^2_1-k^2_2}\int_\Gamma  \left[(k^2_p-k^2_1)k^2_1\gamma_{k_1}(x,y)-(k^2_p-k^2_2)k^2_2\gamma_{k_2}(x,y)\right]\nu_x^\top\nu_y p(y)ds_y\\
&&\quad-\frac{1}{k^2_1-k^2_2}\int_\Gamma \tau_2\left[(k^2_p-k^2_1)\gamma_{k_1}(x,y)-(k^2_p-k^2_2)\gamma_{k_2}(x,y)\right]\tau_1p(y)ds_y.
\enn
Hence,
\ben
&\quad&N_4(p)(x)\\
&=&\rho_f\omega^2\beta\int_\Gamma  \nu _x^\top E_{11}(x,y)\nu_yp(y) ds_y\\
&\quad&+\frac{i\omega\gamma\beta+\rho_f\omega^2(\alpha-\beta)}{(\lambda+2\mu)(k^2_1-k^2_2)}\int_\Gamma  \left[k^2_1\gamma _{k_1}(x,y) - k^2_2\gamma _{k_2}(x,y)\right] \nu_x^\top\nu_y p(y)ds_y\\&\quad&
-\frac{1}{k^2_1-k^2_2}\int_\Gamma  \left[(k^2_p-k^2_1)k^2_1\gamma_{k_1}(x,y)-(k^2_p-k^2_2)k^2_2\gamma_{k_2}(x,y)\right]\nu_x^\top\nu_y p(y)ds_y\\
&\quad&+\frac{\rho_f\omega^2(\alpha-\beta)+i\omega\beta\gamma}{(\lambda+2\mu)(k_1^2-k_2^2)}\int_\Gamma\tau_2\left[\gamma_{k_1}(x,y)-\gamma_{k_2}(x,y)\right]\tau_1p(y)ds_y \\
&\quad&-\frac{1}{k^2_1-k^2_2}\int_\Gamma \tau_2\left[(k^2_p-k^2_1)\gamma_{k_1}(x,y)-(k^2_p-k^2_2)\gamma_{k_2}(x,y)\right]\tau_1p(y)ds_y.
\enn
\end{proof}


\begin{thebibliography}{00}
\bibitem{AKM} K. Ando, H. Kang, Y. Miyanishi, Elastic Neumann-Poincar\'e operators on three dimensional smooth domains: Polynomial compactness and spectral structure, Int. Math. Res. Notices 2019(12) (2019) 3883-3900.
\bibitem{ABD05} X. Antoine, A. Bendali, M. Darbas, Analytic preconditioners for the boundary integral solution of the scattering of acoustic waves by open surfaces, J. Comput. Acoust. 13 (2005) 477-498.
\bibitem{BXY17} G. Bao, L. Xu, T. Yin, An accurate boundary element method for the exterior elastic scattering problem in two dimensions, J. Comput. Phy. 348 (2017) 343-363.
\bibitem{BXY191} G. Bao, L. Xu, T. Yin, Boundary integral equation methods for the elastic and thermoelastic waves in three dimensions, Comput. Method Appl. Methanics Eng. 354 (2019) 464-486.
\bibitem{BT98} M. Benzi, M. Tuma, A sparse approximate inverse preconditioner for nonsymmetric linear systems, SIAM J. Sci. Comput 3 (1998) 968-994.
\bibitem{B41} M.A. Biot, General theory of three-dimensional consolidation, J. Appl. Phys. 12 (2) (1941) 155-164.
\bibitem{B55} M.A. Biot, Theory of elasticity and consolidation for a porous anisotropic solid, J. Appl. Phys. 26 (2) (1955) 182-185.
\bibitem{B561} M.A. Biot, Theory of deformation of a porous viscoelastic anisotropic solid, J. Appl. Phys. 27 (5) (1956) 459-467.
\bibitem{B562} M.A. Biot, Theory of propagation of elastic waves in a fluid-saturated porous solid I. Low-frequency range, J. Acoust. Soc. Am. 28 (2) (1956) 168-178.
\bibitem{B563} M.A. Biot, Theory of propagation of elastic waves in a fluid-saturated porous solid II. Higher frequency range, J. Acoust. Soc. Am. 28 (2) (1956) 179-191.
\bibitem{BET12} O.P. Bruno, T. Elling, C. Turc, Regularized integral equations and fast high-order solvers for sound-hard acoustic scattering problems, Int. J. Numer. Meth. Eng. 91 (2012) 1045-1072.
\bibitem{BG20} O.P. Bruno, E. Garza, A Chebyshev-based rectangular-polar integral solver for scattering by general geometries described by non-overlapping patches, J. Comput. Phys. 421 (2020) 109740.
\bibitem{BL121} O.P. Bruno, S. Lintner, Second-kind integral solvers for TE and TM problems of diffraction by open arcs, Radio Sci. 47 (6) (2012).
\bibitem{BXY192} O.P. Bruno, L. Xu, T. Yin, Weighted integral solvers for elastic scattering by open arcs in two dimensions,  Int. J. Numer. Meth. Eng. 122 (2021) 2733-2750.
\bibitem{BY20} O.P. Bruno, T. Yin, Regularized integral equation methods for elastic scattering problems in three dimensions, J. Comput. Phy. 410 (2020) 109350.
\bibitem{BM71} A. Burton, G. Miller, The application of integral equation methods to the numerical solution of some exterior boundary-value problem, Proc. R. Soc. Lond 323 (1971) 201-210.
\bibitem{CDGS05} B. Carpentieri, I. Duff, L. Giraud, G. Sylvand, Combining fast multipole techniques and an approximate inverse preconditioner for large electromagnetism calculations, SIAM J. Sci. Comput 27 (2005) 774-792.
\bibitem{CD95} J. Chen, G.F. Dargush, Boundary element method for dynamic poroelastic and thermoelastic analysis, Int. J. Solids Struct. 32 (15) (1995) 2257-2278.
\bibitem{CBB91} A.H.D. Cheng, T. Badmus, D.E. Beskos, Integral equation for dynamic poroelasticity in frequency domain with BEM solution, J. Eng. Mech. Asce.  117(5) (1991) 1136-1157.
\bibitem{CN02} S. Christiansen, J.C. N\'{e}d\'{e}lec, A preconditioner for the electric field integral equation based on Calder\'{o}n formulas, SIAM J. Numer. Anal. 40 (3) (2002) 1100-1135.
\bibitem{CK98} D. Colton, R. Kress, Inverse Acoustic and Electromagnetic Scattering Theory, Springer, Berlin, 1998.
\bibitem{B00}  R. de Boer, Theory of porous media, Springer-Verlag, Berlin, 2000.
\bibitem{DR93} G. Degrande, G.De Roeck, An absorbing boundary condition for wave propagation in saturated poroelastic media-Part II: Finite element formulation, Soil Dyn. Earthquake Eng. 12 (1993) 423-432.
\bibitem{DS63} H. Deresiewicz, R. Skalak, On uniqueness in dynamic poroelasticity, Bull. Seismol. Soc. Am. 53 (1963) 783-788.
\bibitem{DE96} S. Diebels, W. Ehlers, Dynamic analysis of a fully saturated porous medium accounting for geometrical and material non-linearities, Int. J. Numer. Methods Eng. 39 (1) (1996) 81-97.
\bibitem{HS20} G. Hsiao, T. Snchez-Vizuet, Time-domain boundary integral methods in linear thermoelasticity, SIAM J. Math. Anal. 52 (2020) 2463-2490.
\bibitem{HW08} G.C. Hsiao, W.L. Wendland, Boundary Integral Equations, Applied Mathematical Sciences, Vol.164, Springer-verlag, 2008.
\bibitem{HXZ14} G. Hsiao, L. Xu, S. Zhang, Solving negative order equations by the multigrid method via variable substitution, J. Sci. Comput 59 (2014) 371-385.
\bibitem{JBB} E. Jimenez, C. Bauinger, O.P. Bruno, IFGF-accelerated integral equation solvers for acoustic scattering, arxiv:2112.06316v2.
\bibitem{KGBB79} V. D. Kupradze, T. G. Gegelia, M. O. Basheleishvili, T. V. Burchuladze, Three-Dimensional Problems of the Mathematical Theory of Elasticity and Thermoelasticity, North-Holland Series in Applied Mathematics and Mechanics, vol. 25, North-Holland Publishing Co., Amsterdam, 1979.
\bibitem{LS98} R.W. Lewis, B.A. Schrefler, The Finite Element Method in the Static and Dynamic Deformation and Consolidation of Porous Media, Wiley, Chichester, 1998.
\bibitem{L14} F. Le Lou\"er, A high order spectral algorithm for elastic obstacle scattering in three dimensions, J. Comput. Phy. 279 (2014) 1-18.
\bibitem{MB89} G.D. Manolis, D.E. Beskos, Integral formulation and fundamental solutions of dynamic poroelasticity and thermoelasticity, Acta Mech. 76 (12) (1989) 89-104.
\bibitem{MS11} M. Messner, M. Schanz, A regularized collocation boundary element method for linear poroelasticity, Comput. Mech. 47 (2011) 669-680.
\bibitem{MS12} M. Messner, M. Schanz, A symmetric Galerkin boundary element method for 3d linear poroelasticity, Acta Mech. 223 (8) (2012) 1751-1768.
\bibitem{N01} J. C. N\'{e}d\'{e}lec, Acoustic and Electromagnetic Equations: Integral Representations for Harmonic Problems, Springer-Verlag, New York, 2001.
\bibitem{S011} M. Schanz, Application of 3D time domain boundary element formulation to wave propagation in poroelastic solids, Eng. Anal. Bound. Elem. 25 (2001) 363-376.
\bibitem{S012} M. Schanz, Wave propagation in viscoelastic and poroelastic continua a boundary element approach, Lecture notes in applied mechanics, Vol.2, Springer-Verlag, 2001.
\bibitem{S09} M. Schanz, Poroelastodynamics: linear models, analytical solutions, and numerical methods, Applied Mechanics Reviews 62 (2009) 030803.
\bibitem{S05} M. Schanz, L. Kielhorn, Dimensionless variables in a poroelastodynamic time domain boundary element formulation, Build. Res. J. 53 (2005) 175-189.
\bibitem{SSU09} M. Schanz, O. Steinbach, P. Urthaler, A boundary integral formulation for poroelastic materials, Proc. Appl. Math. Mech. 9 (1) (2009) 595-596.
\bibitem{XOX19} J. Xie, M.Y. Ou, L. Xu, A discontinuous Galerkin method for wave propagation in orthotropic poroelastic media with memory terms, J. Comput. Phys. 397 (2019) 108865.
\bibitem{YHX17} T. Yin, G.C. Hsiao, L. Xu, Boundary integral equation methods for the two dimensional fluid-solid interaction problem, SIAM J. Numer. Anal. 55(5) (2017) 2361-2393.
\bibitem{ZXY21} L. Zhang, L. Xu, T. Yin, An accurate hypersingular boundary integral equation method for dynamic poroelasticity in two dimensions, SIAM J. Sci. Comput. 43 (2021) 784-810.
\end{thebibliography}
\end{document}